\documentclass[12pt]{article}

\usepackage[T1]{fontenc}
\usepackage[latin1]{inputenc}
\usepackage{bbm}
\usepackage{stmaryrd}
\usepackage{latexsym}
\usepackage{amsfonts}
\usepackage{amsmath,amssymb,amscd,textcomp}
\usepackage{yfonts}
\usepackage{dsfont}
\usepackage{mathabx}
\usepackage[dvips]{graphicx}
\usepackage{epsfig}
\usepackage{psfrag}
\usepackage[font={small}]{caption}
\usepackage{color}
\usepackage{float}
\usepackage{upgreek}
\usepackage{tikz}
\usepackage{tikz-cd}
\usepackage{relsize}
\DeclareFontFamily{U}{txsyc}{}
\DeclareFontShape{U}{txsyc}{m}{n}{
   <-> txsyc%
}{}
\DeclareFontShape{U}{txsyc}{bx}{n}{
   <-> txbsyc%
}{}
\DeclareFontShape{U}{txsyc}{l}{n}{<->ssub * txsyc/m/n}{}
\DeclareFontShape{U}{txsyc}{b}{n}{<->ssub * txsyc/bx/n}{}
\DeclareSymbolFont{symbolsC}{U}{txsyc}{m}{n}
\SetSymbolFont{symbolsC}{bold}{U}{txsyc}{bx}{n}
\DeclareFontSubstitution{U}{txsyc}{m}{n}
\DeclareMathSymbol{\df}{\mathrel}{symbolsC}{"42}
\DeclareMathSymbol{\fd}{\mathrel}{symbolsC}{"43}
\DeclareMathSymbol{\lJoin}{\mathrel}{symbolsC}{"58}
\DeclareMathSymbol{\rJoin}{\mathrel}{symbolsC}{"59}

\DeclareMathOperator{\Ric}{Ric}
\newcommand{\f}[2]{\frac{#1}{#2}}

\newcommand{\cL}{\mathcal{L}}

\newcommand{\cU}{\mathcal{U}}

\newcommand{\EE}{\mathbb{E}}

\newcommand{\NN}{\mathbb{N}}
\newcommand{\PP}{\mathbb{P}}

\newcommand{\RR}{\mathbb{R}}

\newcommand{\ZZ}{\mathbb{Z}}

\newcommand{\fs}{\mathfrak{s}}

\newcommand{\di}{\displaystyle}
\newcommand{\iy}{\infty}

\newcommand{\lt}{\left}
\newcommand{\me}{\medskip}

\newcommand{\ri}{\rightarrow}
\newcommand{\rt}{\right}
\newcommand{\sm}{\smallskip}

\newcommand{\wi}{\widetilde}

\newcommand{\Vol}{\mathrm{Vol}}

\DeclareMathOperator*{\esssup}{ess\,sup}

\newcommand{\fo}{\forall\ }

\newcommand{\lVe}{\lt\Vert}

\newcommand{\rVe}{\rt\Vert}
\newcommand{\sign}{\mathrm{sign}}

\newcommand{\st}{\,:\,}

\newcommand{\un}{\mathds{1}}
\newcommand{\Var}{\mathrm{Var}}

\newcommand{\bq}{\begin{eqnarray*}}
\newcommand{\bqn}[1]{\begin{eqnarray}\label{#1}}
\newcommand{\eq}{\end{eqnarray*}}
\newcommand{\eqn}{\end{eqnarray}}
\newcommand{\wwtbp}{\par\hfill $\blacksquare$\par\me\noindent}

\newcommand{\thistitlepagestyle}{}

\newcommand{\ttsim}{\raise.17ex\hbox{$\scriptstyle\mathtt{\sim}$}}

\newcommand{\kh}{\kern .08em}

\newtheorem{pro}{Proposition} 
\newtheorem{cor}[pro]{Corollary}

\newtheorem{theo}[pro]{Theorem}

\renewcommand{\thepro}{\arabic{pro}}

\newenvironment{deff}
{\par\me\refstepcounter{pro}\noindent{\bf Definition \thepro\ }}
{\par\hfill $\square$\par\sm\noindent}

\newenvironment{rem}
{\par\me\refstepcounter{pro}\noindent{\bf Remark \thepro\ }}
{\par\hfill $\square$\par\sm\noindent}

\newcommand{\proof}{\par\me\noindent\textbf{Proof}\par\sm\noindent}
\newcommand{\prooff}[1]{\par\me\noindent\textbf{#1}\par\sm\noindent}

\newcommand{\comment}[1]{}

\allowdisplaybreaks

\title{On the separation cut-off phenomenon\\ for Brownian motions on high dimensional rotationally symmetric compact manifolds}
 \author{Marc Arnaudon${}^{(1)}$, Kol\'eh\`e Coulibaly-Pasquier${}^{(2)}$ and Laurent Miclo${}^{(3)}$
 }

 \date{\vbox{\copy0
 \copy1
 \copy2
}
 }

\begin{document}

\setbox0=\vbox{
\large
\begin{center}
${}^{(1)}$ Institut de Math\'ematiques de Bordeaux, UMR 5251\\
Universit\'e de Bordeaux, CNRS and INP Bordeaux 
\end{center}
} 
\setbox1=\vbox{
\large
\begin{center}
${}^{(2)}$ Institut \'Elie Cartan de Lorraine, UMR 7502\\
Universit\'e de Lorraine and CNRS
\end{center}
} 
\setbox2=\vbox{
\large
\begin{center}
${}^{(3)}$ Toulouse School of Economics, UMR 5314\\
Institut de Math\'ematiques de Toulouse, UMR 5219\\
CNRS and Universit\'e de Toulouse
\end{center}
} 
\setbox4=\vbox{
\hbox{marc.arnaudon@math.u-bordeaux.fr\\[1mm]}
\hbox{Institut de Math\'ematiques de Bordeaux\\}
\hbox{F.\ 33405, Talence, France}
}
\setbox5=\vbox{
\hbox{kolehe.coulibaly@univ-lorraine.fr\\[1mm]}
\hbox{Institut \'Elie Cartan de Lorraine\\}
\hbox{Universit\'e de Lorraine}
}
\setbox6=\vbox{
\hbox{miclo@math.cnrs.fr\\[1mm]}
\hbox{Toulouse School of Economics\\}
\hbox{1, Esplanade de l'université\\}
\hbox{31080 Toulouse cedex 06, France\\
}
}

\maketitle
\thistitlepagestyle
\abstract{Given a family of rotationally symmetric compact manifolds indexed by the dimension and a weight function, the goal of this paper is to investigate the cut-off phenomenon for the Brownian motions on this family. We provide a class of  compact manifolds with non-negative Ricci curvatures for which the cut-off in separation with windows occurs, in high dimension,  with different explicit mixing times. We also produce counter-examples, still with non-negative Ricci curvatures, where there are no cut-off in separation. In fact we show  a phase transition for the cut-off phenomenon concerning the Brownian motions on a rotationally symmetric compact manifolds.
 Our proof is based on a previous construction of a sharp strong stationary times by the authors, and some quantitative estimates on the two first moments of the covering time of the dual process. The concentration of measure phenomenon for the above family of manifolds appears to be relevant for the study of the corresponding cut-off.}

\vfill\null
{\small
\par\vskip.3cm
 \textbf{Fundings:} the grants ANR-17-EURE-0010 and AFOSR-22IOE016  are acknowledged by LM.
}\par

\newpage

\section{Introduction}
\subsection{Overview}
The main purpose of the present paper is to investigate the cut-off phenomenon in separation for Brownian motion on high  dimensional compact manifolds, especially for model space that are the rotationally symmetric compact manifolds. 
In the context of card shuffling, the cut-off phenomenon was discovered by Diaconis and Shahshahani  \cite{DS} and Aldous and Diaconis \cite{AD}. Cut-off phenomenon is an abrupt transition from out of equilibrium to equilibrium, which occurs for certain Markov processes, when the size of the state space become large. Afterward, the cut-off phenomenon has been proven for a large variety of finite Markov chains, see e.g.\ Diaconis \cite{MR1374011}, Diaconis and Fill \cite{MR1071805}, Levin, Peres and Wilmer \cite{MR2466937} and Ding, Lubetzky and Peres \cite{zbMATH05659491}.  Nevertheless the literature on the cut-off phenomenon for Markov processes  on a continuous state space is rather sparse. For example  Saloff-Coste \cite{MR1306030} has proven the cut-off phenomenon  in total variation distance for the Brownian motions on the spheres $ \mathbb {S}^n$ for high dimensions $n$, with a mixing time of order $\ln(n)/(2n)$, see also M\'eliot \cite{MR3201989} for extensions to  classical symmetric spaces of compact type. Their approach are based on complete knowledge of the spectral decomposition. It is shown in Hermon, Lacoin and  Peres \cite{zbMATH06618510} that total variation and separation cut-off are not equivalent and neither one implies the other. In a precedent paper \cite{sphere} we have shown that the cut-off in separation also occurs for the Brownian motion on the sphere of high dimensions $n$ with a mixing time of order $\ln(n)/n$. In the present paper we generalize such a result for a large class of manifold, and  as example we strengthen this result on spheres with a cut-off in separation with windows. Note that controlling the separation discrepancy is essentially (but not exactly) a $L^{\infty}$ control while the control of the total variation is a $L^1$ control. Heuristically,  the difference in the mixing times in total variation and separation comes from the fact that $L^1 $ estimates only require  the dual process to see a big part of the volume, and by concentration of measure phenomenon it is sufficient to see the ``equator'', while to get $ L^{\infty}$ estimates we have wait for the dual process to cover all the sphere, namely to reach the opposite pole, and this takes twice as long.  
 
 Our goal here is to check that there is a cut-off phenomenon in separation with windows for a large class of family of rotationally symmetric manifolds with non-negative Ricci curvature Theorems \ref{cut-off-ge} and  \ref{cut-off-ge_alpha}, including the case of spheres Corollary \ref{sphere}. We also give examples of  rotationally symmetric manifolds  with non-negative Ricci curvature where there is no cut-off in separation Theorems \ref{nocut} and \ref{cut-off-ge_alpha2}.  In fact we show  a phase transition for the cut-off phenomenon concerning the Brownian motions see Theorem \ref{theo1}. Our results are connected with those of Salez, concerning  sequences of irreducible Markov chains with symmetric support  and non-negative coarse Ricci curvature that exhibit cut-off in total variation when an additional product condition hypothesis is satisfied, see \cite{Salez} for the precise statement. 
 
  Our proof is based on two ingredients, the resort to the strong stationary times for $X_n$ presented in \cite{arnaudon:hal-03037469}
and the detailed quantitative estimates on the cover time of dual process (see \cite{zbMATH07470497}) that appear to be an one-dimensional diffusion processes in the case of rotationnaly symmetric manifolds. The concentration of volume phenomenon plays a crucial role to detect the scale    on which the cut-off phenomenon occurs.  This alternative point of view differs from the traditional approach based on spectral analysis and could  be extended to other situations where spectral information is less available.

\subsection{Geometric framework}
For $n\geq 2$, let $M_f^{n}$ be the product manifold
$ [0,L] \times \mathbb{S}^{n-1}/ \sim$, where $(r_1, \theta_1) \sim (r_2, \theta_2) $ if $(r_1, \theta_1) = (r_2, \theta_2) $ or $r_1 = r_2 = 0 $ or $r_1 = r_2 = L $,  endowed with the warping product metric  $$ ds^2 = dr \otimes dr + f^2(r) d\theta \otimes d\theta  ,$$ where $\mathbb{S}^{n-1} $ is the usual sphere of dimension $n-1$ and radius $1$, $  d\theta \otimes d\theta $ is the standard metric on the sphere and $ f $ is a regular real function that satisfies the following assumption:
\bqn{H_f} 
\lt\{\begin{array}{rcl}
f : [0,L] &\to& \mathbb{R}_{+},\\
 f(s) \sim_{0} s&,& f(L-s)\sim_{0} s \\
f^{(2k)}(0) &=& f^{(2k)}(L) = 0 ,  k \in \mathbb{Z}_{+}\\
\end{array}
\right.
\eqn
We will call such function a {\bf{weight}} function, we will assume all along the paper that $f$ is a weight function.
Later, further conditions will be required to ensure the regularity of the metric at $\wi 0\sim (0,.) $ and $\wi L \sim (L,.)$.
 The volume of the geodesic ball $B(\wi 0,r) $ in $M_f^{n}$ centered at $\wi 0$  of radius $r\in [0, L]$ is given by 
$$  \Vol_n(B(\wi 0,r)) = c_n \int_0^r f^{n-1}(s) ds ,$$
where $c_n =\frac{2\pi^{n/2}}{\Gamma(\frac{n}2)}   $ is the volume of $ \mathbb{S}^{n-1}$. The area of the geodesic sphere $\partial B(\wi 0,r) $ is $ c_n f^{n-1} (r) $ and the mean curvature of any point in $\partial B(\wi 0,r) $ is given by $(n-1)\frac{f'(r)}{f(r)}$. We have $\Ric(v) = \lt( (n-2)\frac{1-f'(r)^2}{f^2(r)} - \frac{f''(r)}{f(r)} \rt) v $  if $ v \in T \mathbb{S}^{n-1} $ and $\Ric(\partial_r) = \lt( -(n-1) \frac{f''(r)}{f(r)} \rt) \partial_r$, where $ \Ric$ denote the Ricci tensor. For a good introduction to warped products, see Chapter 3 in Petersen \cite{Petersen}.
\par\sm
Here is our main object of interest.
\begin{deff}\label{BM}
For any $n\in\NN\setminus\{1\}$,  $X_n\df(X_n(t))_{t\geq 0}$ 
 stands for the Brownian motion on  $M_f^{n}$ started at $\wi 0 $ and time-accelerated by a factor $2$,  i.e.\ the $ \Delta$-diffusion in $M_f^{n}$.
So the generator of $X_n$ is the Laplacian $ \Delta$ and not the Laplacian divided by 2 as it is sometimes more usual in Probability Theory. % To lighten the exposition, we will sometimes drop the subscript $n$ and simply write $X_n$. 
 \end{deff}
It was seen in \cite{zbMATH07470497} that $X_n$ can be intertwined with the dual process $D\df(D(t))_{t\geq 0}$ taking values in the closed balls of $M_f^{n}$ centered at $\wi 0$, starting at $\{\wi 0\}$ and absorbed in finite time $\tau_n$ in the whole set $M_f^{n}$. In  \cite{arnaudon:hal-03037469}, several couplings of  $X_n$ and $D$ were constructed, so that for any time $t\geq 0$, the conditional law of $X_n(t)$ knowing the trajectory $D({[0,t]})\df(D(s))_{s\in[0,t]}$ is the normalized uniform law over $D(t)$, which will be denoted $\Lambda(D(t),\cdot)$ in the sequel. Furthermore, $D$ is progressively measurable with respect to $X_n$, in the sense that for any $t\geq 0$, $D({[0,t]})$ depends on $X_n$ only through $X_n({[0,t]})$.
Due to these couplings and to general arguments from Diaconis and Fill \cite{MR1071805}, $\tau_n$ is a strong stationary time for $X_n$, meaning that $\tau_n$ and $X_n(\tau_n)$ are independent and $X_n(\tau_n)$ is uniformly distributed over  $M_f^{n}$.
As a consequence we have 
\bq
\fo t\geq 0,\qquad \fs(\cL(X_n(t)),\cU_{n})&\leq & \PP[\tau_n>t]\eq
where  the l.h.s.\ is the separation discrepancy between the law of $X_n(t)$ and the uniform distribution $\cU_{n}$ over $M_f^{n}$. Notice that $\cU_{n} (B(\wi 0,r)) = \frac{\int_0^r f^{n-1}(s) ds} {\int_0^L f^{n-1}(s) ds}$ for any $r\in[0,L]$.\par
Recall that the separation discrepancy between two probability measures $\mu$ and $\nu$ defined on the same measurable space
is given by
\bq
\fs(\mu,\nu)&=&\esssup_{\nu} 1-\f{d\mu}{d\nu}\eq
where $d\mu/d\nu$ is the Radon-Nikodym density of $\mu$ with respect to $\nu$. Note that $\Vert \mu- \nu \Vert_{\mathrm{tv}} \le \fs(\mu,\nu)$, where $\Vert \cdot \Vert_{\mathrm{tv}} $ stands for the total variation.
\par
\begin{rem}
Note that for any $t\in [0,\tau_n)$, the ``opposite pole'' $\wi L$ does not belong to the support of $\Lambda(D(t),\cdot)$.
It follows from an extension of Remark 2.39 of Diaconis and Fill \cite{MR1071805} that $\tau_n$ is even a sharp  strong stationary time for $X_n$, meaning
that
\bq
\fo t\geq 0,\qquad \fs(\cL(X_n(t)),\cU_{n})&= & \PP[\tau_n>t]\eq
\par
Thus the understanding of the convergence in separation of $X_n$ toward $\cU_{n}$ amounts to understanding the distribution of $\tau_n$.

\end{rem}

\subsection{Cut-off phenomenon}\label{cutoff}

For fixed $n$, the Brownian motion $X_n$ in $M_f^n$ converges in law to $ \mathcal{U}_n $, namely
$$X_n(t)  \overset{\mathcal{L}}{\to}_{t \to +\infty}  \mathcal{U}_n.$$ 
Quantifying this convergence to equilibrium is relevant when the dimension $n$ becomes large. This speed of convergence or mixing time, depends one the way the difference between the time marginal and the uniform distribution is measured.  A cut-off phenomenon in separation at time $a_n$ is a kind of phase transition, namely the separation discrepancy between $X_n $ and the equilibrium abruptly drops from the largest value $1$ to the smallest one $0$ on a small interval around $ a_n$.
More precisely, we say that
the family of diffusion processes $(X_n)_{n\in\NN\setminus\{1\}}$ has a cut-off in separation with mixing times $(a_n)_{n\in\NN\setminus\{1\}} $ and windows $(b_n)_{n\in\NN\setminus\{1\}} $ when

\bq
\forall n\ge 1, \quad 0<b_n\le a_n,
\eq
\bq
\fo r>0,\qquad\lim_{n\ri\iy} \fs(\cL(X_n(a_n+rb_n)),\cU_{n}) = \lim_{n\ri\iy}\PP\lt[\tau_n>a_n+rb_n\rt]&=&0\\
\fo r\in(0,1),\qquad \lim_{n\ri\iy} \fs(\cL(X_n(a_n-rb_n)),\cU_{n})=1-\lim_{n\ri\iy}\PP\lt[\tau_n \leq a_n-rb_n\rt]&=&1.\eq

When $\forall n\ge 1$, $b_n=a_n$, we simply say that the family of diffusion processes $(X_n)_{n\in\NN\setminus\{1\}}$ has a cut-off in separation with mixing times $(a_n)_{n\in\NN\setminus\{1\}} $.

\subsection{Intertwining relations}

Writing $B(\wi 0,R(t))\df D(t)$ for $t\in[0,\tau_n]$, it has been seen in \cite{zbMATH07470497} that $R\df(R(t))_{t\in[0,\tau_n]}$ is solution to the stochastic differential equation
\bqn{R}
\fo t\in(0,\tau_n),\qquad dR(t)&=&\sqrt{2} dB(t)+b_n(R(t)) dt\eqn
and 
\bqn{taun}
\tau_n&=& \inf\{t\geq 0\st R(t)=L\}\eqn
where $(B(t))_{t\geq 0}$ is a standard Brownian motion in $\RR$ and the mapping $b_n$ is given by
\bqn{bn}
\fo r\in(0,L),\qquad
b_n(r)&\df&2\f{f^{n-1}(r)}{\int_0^r f^{n-1}(u)\,du}-(n-1)\f{ f'(r)}{f(r)}
\eqn
It is not difficult to check  that as $r$ goes to $0_+$
\bq
b_n(r)&\sim &\f{n+1}{r}\eq
and this is sufficient to insure that 0 is an entrance boundary for $R$, so that starting from 0, it will never return to 0 at positive times.
\\

\par\sm
In the following  corollary we explicit two intertwining relations, which were constructed in \cite{arnaudon:hal-03037469} Theorems 3.5 and 4.1, enabling to deduce $\tau_n$ from the Brownian motion $X_n$ (and independent randomness for the second construction):
\begin{cor}\label{cor1}
Consider the Brownian motion $X_n\df (X_n(t))_{t\geq 0}$ in $M_f^{n}$ described in Definition \ref{BM}. For $x\in M_f^{n} \backslash\{\wi 0, \wi L\}$, denote by $N(x)$ the unit vector at $x$ normal to the sphere centred at $\wi 0$ with radius $\rho(\wi 0,x)$ where $\rho $ is the distance in $M_f^{n}$, pointing towards $\wi 0$: $N(x)=-\nabla\rho(\wi 0,\cdot)(x)$.
\begin{itemize}
\item[(1)] \textbf{Full coupling}.
 Let $D_1(t)$ be the ball in $M_f^{n}$ centred at $\wi 0$ with radius $R_1(t)$ solution started at $0$ to the It\^o equation
\bq
dR_1(t)&=&-\sqrt{2} \langle N(X_n(t)), dX_n(t))\rangle + n\left[2\frac{f'}{f}(\rho(\wi 0,X_n(t)))-\frac{f'}{f}(R_1(t))\right]\, dt
\eq
This evolution equation is considered up to  the hitting time  $\tau_n^{(1)}$ of $L$ by $R_1(t)$. 
\item[(2)] \textbf{Full decoupling, reflection of $D$ on $X_n$}.
 Let $D_2(t)$ be the ball in $M_f^{n}$ centered at $\wi 0$ with radius $R_2(t)$ solution started at $0$ to the It\^o equation
\bq
dR_2(t)&=&-\sqrt{2} dW_t +2dL_t^{R_2}[\rho(\wi 0,X_n)]- n\frac{f'}{f}(R_2(t))\, dt
\eq
where $(W_t)_{t\ge 0}$ is a real-valued Brownian motion independent of $(X_n(t))_{t\ge 0}$ and $(L_t^{R_2}[\rho(\wi 0,X_n)])_{t\in[0,\tau_n^{(2)}]}$ is the local time  at $0$ of the process $R_2-\rho(\wi 0,X_n)$. These considerations are valid up to  the hitting time  $\tau_n^{(2)}$ of $L$ by $R_2(t)$.
\end{itemize}
Let $D(t)$ be the ball  in $M_f^{n}$ centered at $\wi 0$ with radius $R(t)$, defined in~\eqref{R}, and let $\tau_n$ be the stopping time defined in~\eqref{taun}.

Then we have: 
\begin{itemize}
\item[(1)]
 for $i=1,2$ $X_n(\tau_n^{(i)})$ is uniformly distributed in $M_f^{n}$,  
\item[(2)] the pairs  $(\tau_n^{(1)}, (D_1(t))_{t\in [0,\tau_n^{(1)}]})$,  $(\tau_n^{(2)}, (D_2(t))_{t\in [0,\tau_n^{(2)}]})$ and  $(\tau_n, (D(t))_{t\in [0,\tau_n]})$ have the same law. In particular $\tau_n^{(1)}$ and $\tau_n^{(2)}$ satisfy Proposition~\ref{mean-var} and Theorem \ref{th1} below.
\end{itemize}
\end{cor}
 \par\me

\subsection{Outline of the paper and main result}

 The paper is organized as follow in Section 2 we compute the Green functional of the one-dimensional diffusion associated to the radius of the dual process, and we give  a tractable formulation of all moments of $ \tau_n$, the covering time of the dual process. In Section 3, we compute the mixing time for several rotationnaly symmetric manifold and we show that depending on the shape of the weight function $f$, the  cut-off occurs in separation for the Brownian motion on $ M_f^n$ for high dimensions $n$, see  Theorem  \ref{cut-off-ge} and \ref{cut-off-ge_alpha}, or there is no cut-off in separation, see Theorem  \ref{nocut} and \ref{cut-off-ge_alpha2}. \par
 
 These results are essentially summarized by the following theorem, showing a phase transition (with respect to the parameter  $ \alpha \in (-1,+\iy)$ introduced below) for the cut-off phenomenon concerning the Brownian motions on the model $M^n_f$ for high dimensions $n$,  depending on the shape of the function $f$ at $L/2$.
 Let us first introduce, in order to simplify the exposition, another set of assumptions on $f$:
\bqn{H_f2} 
\lt\{\begin{array}{c}
\fo s\in[0,L], \qquad f(L-s) = f(s),\\
\fo s\in [0,L/2),\qquad f'(s)>0,\\
\fo s\in [0,L]\setminus \{L/2 \},\qquad f''(s)\leq 0,
\end{array}
\right.
\eqn

\begin{theo}\label{theo1}
Consider a   $C^2$ function $f$ on $[0, L]\setminus \{L/2\}$ and $C^1$ in  $[0, L]$, satisfying Assumptions \eqref{H_f} and \eqref{H_f2}.
Assume there exist $ \alpha \in (-1,+\iy)$ and $C >0$ 
such that for $h\neq 0$ small enough,
 \bqn{fsec}
 f''(L/2- h)&=& -{C}\vert h \vert^{\alpha } + o(\vert h \vert^{\alpha })  \eqn
 Let  $X_n\df(X_n(t))_{t\geq 0}$ be the Brownian motion described in Definition \ref{BM}.
\begin{itemize}
\item if  $\alpha \in (-1,0)$ then $ (X_n)_{n\in \NN\setminus\{1\}}$ has a cut-off in separation at time $ C_1/n$, with
 $$C_1 =  2 \int_0^{L/2} \frac{f(s)}{f'(s)}, $$
\item if  $\alpha = 0$  then $ (X_n)_{n\in \NN\setminus\{1\}} $ has a cut-off in separation at time $ C_2\ln(n)/n$,
with
$$C_2 =  \frac{f(L/2)}{C}, $$
\item if  $\alpha > 0$ then  $(X_n)_{n\in \NN\setminus\{1\}}$ has no cut-off in separation,
\end{itemize}
\end{theo}
\par
An instance where \eqref{fsec} is satisfied is when there exist $ \alpha \in (-1,+\iy)$, $C >0$ and
 $\epsilon\in (0,L/2)$ such that
  \bq
\fo h\in[-\epsilon,\epsilon],\qquad
 f(L/2 + h) &=& f(L/2) -C\vert  h \vert^{2+\alpha}.  \eq
\par
 Note the additional factor $\ln(n)$ at the critical case $\alpha=0$ for the phase transition. We conjecture that in the supercritical cases $\alpha>0$, $\tau_n/\EE[\tau_n]$ converges in distribution for large $n$ toward a particular law depending on $\alpha$,  that would reflect the fact that the larger $\alpha>0$, the more difficult is the mixing. To go toward this result, we should investigate more moments of the strong stationary times $\tau_n$ than just the two first ones, as we will do below.
\par
%An appendix shows that a Riemannian manifold admitting a representation of the form $M^n_{f_x}$ viewed from any of its point $x$, where $f_x$ is a function satisfying \eqref{H_f}, is necessarily a sphere.

\section{Preliminary  results}\label{green}

Define for any $r\in[0,L]$,
\bq
 I_n(r)&\df& \int_0^r f^{n-1}(s) ds\\
  b_n(r) &\df& \f{d}{dr} \ln \lt(\frac{I_n^2(r)}{f^{n-1}(r)}\rt)
  \eq
Let $L_n := \partial^2_r + b_n(r) \partial_r $ be the generator of  $ R $   defined in \eqref{R}.
Here is our first preliminary result:
\begin{pro}\label{Poisson}
Given $ g \in C_b([0, L])$, the bounded solution $\phi_n$ of the Poisson equation
\bq
 \lt\{\begin{array}{rcl}
  L_n \phi_ n &=& -g  \\ 
  \phi_n (L) &=& 0 \\
\end{array}\rt.\eq
is given by:
 \bqn{phinr}
 \fo r\in[0,L],\qquad
  \phi_n (r) &=&  \int_r^L  \frac{f^{n-1}(t)}{I_n^2(t)}  \lt(\int_0^t  \frac{I_n^2(s)}{f^{n-1}(s)} g(s) ds \rt) dt.  
\eqn
    So the Green operator $G_n$ associated to $L_n $ is given by 
    \bq 
    \fo g\in C_b([0, L]),\,\fo r\in[0,L],\qquad
    G_n[g](r) = \int_r^L  \frac{f^{n-1}(t)}{I_n^2(t)}  \lt(\int_0^t  \frac{I_n^2(s)}{f^{n-1}(s)} g(s) ds \rt) dt.
    \eq
 \end{pro}
   
   \begin{proof}
    
 Use the Remark \ref{rem-int} below to justify integrability of $$ [0,L]\ni  t \mapsto\frac{f^{n-1}(t)}{I_n^2(t)}   \int_0^t  \frac{I_n^2(s)}{f^{n-1}(s)} g(s) ds $$ at $ 0$ and $ L$.
For the function defined in \eqref{phinr}, we clearly have, $\phi_n (L) =0 $, and for any $r\in[0,L]$,
   $$ \phi_n' (r) = - \frac{f^{n-1}(r)}{I_n^2(r)} \int_0^r  \frac{I_n^2(s)}{f^{n-1}(s)} g(s) ds  $$
   $$ \phi_n ''(r) = - \lt(\frac{f^{n-1}}{I_n^2}\rt)' (r) \int_0^r  \frac{I_n^2(s)}{f^{n-1}(s)} g(s) ds  - g(r) .$$
 It follows that  \begin{align*} 
   L_n \phi_ n (r) &= -g(r) - \lt(\frac{f^{n-1}}{I_n^2}\rt)'(r) \int_0^r  \frac{I_n^2(s)}{f^{n-1}(s)} g(s) ds\\
   &+ \lt( \ln \frac{I_n^2}{f^{n-1}}\rt)'(r) \phi_n' (r) \\
   &= -g(r) - \lt(\frac{f^{n-1}}{I_n^2}\rt)' (r)\int_0^r  \frac{I_n^2(s)}{f^{n-1}(s)} g(s) ds\\
   &-  \lt(  \ln \frac{f^{n-1}}{I_n^2} \rt)'(r) \lt(- \frac{f^{n-1}(r)}{I_n^2(r)} \int_0^r  \frac{I_n^2(s)}{f^{n-1}(s)} g(s) ds  \rt) \\
   &=-g(r).\\
   \end{align*}
    
   \end{proof}

   \begin{rem}\label{rem-int}
   
Let us show that the  integral $  \int_0^{L}  \frac{f^{n-1}(t)}{I_n^2(t)}  \lt(\int_0^t  \frac{I_n^2(s)}{f^{n-1}(s)}  ds \rt) dt$ is finite. Since $ f(s) \sim_{s \ri 0_+} s$ we have $I_n(s) \sim_{s \ri 0_+} \frac{s^n}{n} , $ so $\frac{I_n^2(s)}{f^{n-1}(s)} \sim_{s \ri 0_+ } \frac{s^{n+1}}{n^2}$ hence $ \frac{f^{n-1}(t)}{I_n^2(t)}  \int_0^t  \frac{I_n^2(s)}{f^{n-1}(s)}  ds $ is integrable at $0$.
\par
Concerning the integrability at $L$, since $I_n(L)$ is positive and finite, it is sufficient to see that  for $\varepsilon\in (0,L)$, $ \int_\varepsilon^{L} f^{n-1}(t) \lt(\int_\varepsilon^t  \frac{1}{f^{n-1}(s)}  ds \rt) dt$ is finite and this is indeed true since $f(L-s) \sim_{s\ri 0_+} s $.\par
The above considerations further enable us to see
that 
\bq
\lim_{t\ri 0_+} \frac{1}{I_n(t)}  \int_0^t  \frac{I_n^2(s)}{f^{n-1}(s)}  ds  &=&0\eq
 justifying the following integration by parts:
\begin{align*}
\int_0^{L}  \frac{f^{n-1}(t)}{I_n^2(t)}  \lt(\int_0^t  \frac{I_n^2(s)}{f^{n-1}(s)}  ds \rt) dt 
&=\int_0^{L}  \lt(-\frac{1}{I_n(t)}\rt)'  \lt(\int_0^t  \frac{I_n^2(s)}{f^{n-1}(s)}  ds \rt) dt \\
&= - \frac{1}{I_n(L)}  \int_0^L  \frac{I_n^2(s)}{f^{n-1}(s)}  ds  +  \int_0^L  \frac{I_n(s)}{f^{n-1}(s)}  ds \\
&=   \int_0^L  \frac{I_n(s)}{f^{n-1}(s)} \lt( \frac{I_n(L)-I_n(s)}{I_n(L)}\rt)  ds \\
&= \frac{1}{\Vol_n(M)}\int_ 0^L \frac{\Vol_n(B(\wi 0,s))\Vol_n(B^c(\wi 0,s))}{\Vol_{n-1}(\partial B(\wi 0,s))} ds
\end{align*}
where $\Vol_{n-1}$ is the $(n-1)$-dimensional Hausdorff measure.
The last r.h.s.\ and the following  Proposition \ref{mean-var}  show that
 $\EE[\tau_n] \le \frac{L}{h_n(M_f^n)}$, with the Cheeger constant $h_n(M) \df\inf_{D \subseteq M, \Vol_n( D) \le \Vol_n(M) / 2} \frac {\Vol_{n-1}(\partial D)}{\Vol_{n}(D)} $.
    \end{rem}   

Let $u_{n,0}\df \un $, the constant function taking the value 1 on $[0,L]$, and consider the following sequence  $(u_{n,k})_{k\in\NN} $, defined inductively by bounded solution of

 \bqn{def-u}
\fo k\in\NN,\qquad \lt\{\begin{array}{rcl}
 L_n u_{n,k} &=& -k   u_{n,k-1}\\
    u_{n,k} (L) &=& 0 .\\
\end{array}\rt.\eqn

 We have for all $n,k \in \mathbb{Z}^+$,  $$ \frac{u_{n,k}}{k !} =  G_n^{\circ k}[\un]\df G_n[G_n[\cdots [G_n [\un]]\cdots]]. $$

\begin{pro}\label{mean-var}
We have for all $ n \ge 2 $:
\bqn{compl0} \nonumber\EE[\tau_n ] &=& u_{n,1}(0)\, =\,  \int_0^L  \frac{f^{n-1}(t)}{I_n^2(t)}  \lt(\int_0^t  \frac{I_n^2(s)}{f^{n-1}(s)}  ds \rt) dt\\
&=& \int_0^L  \frac{I_n(s)}{f^{n-1}(s)} \lt( \frac{I_n(L)-I_n(s)}{I_n(L)}\rt)  ds
\eqn
\bq \EE[\tau_n^2 ] &=& u_{n,2}(0)\, = \, 2 \int_0^L  \frac{f^{n-1}(t)}{I_n^2(t)}  \lt(\int_0^t  \frac{I_n^2(s)}{f^{n-1}(s)} u_{n,1}(s) ds \rt) dt,
 \eq
and 
more generally, for any $k\in\ZZ_+$,
 \bq  \EE[\tau_n^{k}] &=& u_{n,k}(0)\, = k! G_n^{\circ k}[\un] (0)  \eq
\end{pro}
\proof
Suppose by induction that $ u_{n,k}(x)  =  \EE_x [\tau_n^k] .$ This is clearly satisfied for $k=0$.
Using It\^o's formula, we have for all $ 0\le t \leq  \tau_n$, for the process $R$ defined in \eqref{R} and starting with $R(0)=x\in[0,L]$,
$$ u_{n,k+1}(R(t)) - u_{n,k+1}(x) = - (k+1)\int_0^t u_{n,k} (R(s))  ds + M_t $$
where $ (M_t)_{t\in[0,\tau_n]}$ is a martingale. Consider this equality with $t=\tau_n$, take expectation and use the Markov property to get
 \begin{align*}
u_{n,k+1}(x)  &=  (k+1)\EE_x \lt[\int_0^{\tau_n} u_{n,k}(R(s)) ds\rt] \\
             &= (k+1)\EE_x \lt[\int_0^{\tau_n} \EE_{R(s)}[\tau_n^k] ds\rt] \\
             &= (k+1)  \EE_x \lt[\int_0^{\tau_n} (\tau_n - s)^k ds\rt] = \EE_x[\tau_n^{k+1}] .\\
\end{align*}
\par
%Proposition \ref{Poisson} and Remark \ref{rem-int} lead to the other results.
From Remark \ref{rem-int}, $G_n(g)(r)$ is defined and bounded if $g$ is bounded. Moreover $G_n(g)(0)\ge G_n(g)(r)$ if $g\ge 0$. This implies that if $u_{n,k}(0)<\infty$, then $u_{n,k+1}$ is defined and bounded.
\wwtbp

The following characterisation of the cut-off phenomenon holds in general and in particular for the Brownian motion in $M_f^n $ with initial value $ \wi 0 $.
The underlying idea of comparing the variance and the square of the expectation of sharp strong stationary times was also used by Diaconis and Saloff-Coste \cite{MR2288715}.
\par
As usual, we say that  $ f_n = o(g_n)$ when  $\frac{f_n}{g_n} \to 0$ as $n$ goes to infinity, and  $ f_n = O(g_n)$ when there exists a constant $c$ such that $  f_n \le c g_n  $.\\
Let  $ a_n  \sim_{n \ri \infty}\EE[\tau_n] =\int_0^L  \frac{f^{n-1}(t)}{I_n^2(t)}  \lt(\int_0^t  \frac{I_n^2(s)}{f^{n-1}(s)}  ds \rt)dt $.
\begin{theo}\label{th1}

Suppose that for some sequence $(b_n)_{n\ge 1}$ we have $\forall n\ge 1$, $ 0< b_n\le a_n$, 
$$ a_n  -\EE[\tau_n] =o(b_n) \quad \hbox{and}\quad  \Var(\tau_n) = o(b_n^2), $$
 then the family of diffusion processes $(X_n)_{n\in\NN\setminus\{1\}}$ has a cut-off in separation with mixing times $(a_n)_{n\in\NN\setminus\{1\}} $ and windows  $(b_n)_{n\in\NN\setminus\{1\}} $   in the sense of Section \ref{cutoff}.

\end{theo} 
\proof
Since $\tau_n$ is  a sharp  strong stationary time for $X_n$, we have 
\bq\fo t\geq 0,\qquad \fs(\cL(X_n(t)),\cU_{n}) &= &\PP\lt[\tau_n>t\rt].\eq
 Using Bienaym\'e-Tchebychev inequality we have for any $r>0$,
\bq
\PP\lt[\tau_n>a_n+rb_n\rt] &=& \PP\lt[\tau_n- \EE[\tau_n] > rb_n + a_n - \EE[\tau_n ]\rt] \\ 
&\le& \frac{\Var(\tau_n)}{(rb_n + a_n - \EE[\tau_n] )^2}\\ 
& = &\frac{o(b_n)^2}{(rb_n + o(b_n)  )^2}\ =\ o(1)\\  
\eq
\par
For the behavior before $a_n$, write for $r\in(0,1)$,
\bq
\PP[\tau_n\leq a_n-rb_n]&=&\PP[\tau_n-\EE[\tau_n]\leq a_n-rb_n-\EE[\tau_n]]\eq
and note that $a_n-rb_n-\EE[\tau_n]=-r(1+ o(1))b_n<0$ for $n$ large. Thus we get
\bq
\PP[\tau_n\leq a_n-r_nb]&\leq & \frac{\Var(\tau_n)}{(rb_n + o(b_n)  )^2} \ =\ o(1)\eq
\wwtbp

Proposition \ref{mean-var} could be used to compute the variance of $ \tau_n$, but it is not well-adapted to compute an equivalent, so let us give an alternative computation of the variance.

\begin{pro} \label{var}The variance of $\tau_n$ is given by   $$\Var(\tau_n) =  2\int_0^{L} \frac{f^{n-1}(t)}{I_n^2(t)} \int_0^t \frac{I_n^2(s)} {f^{n-1}(s)} (u_{n,1}'(s))^2 ds$$
\end{pro} 

\proof
Recall that $u_{n,1} $ is the solution of 
\bq
 \lt\{\begin{array}{rcl}
 L_n u_{n,1} &=& -1\\
    u_{n,1} (L) &=& 0 \\
\end{array}\rt.\eq
Using \eqref{R}  and It\^o formula, we have:
$$ u_{n,1}(R({\tau_n})) - u_{n,1}(0) = - \tau_n + \sqrt{2} \int_0^{\tau_n}u'_{n,1}  (R(s))dB_s, $$
and since $ u_{n,1}(0) = \EE[\tau_n]$, we have 
$$ \Var( \tau_n)   = 2 \EE\lt[ \int_0^{\tau_n}(u'_{n,1})^2  (R(s))ds \rt] .$$
Let $\phi_n $ be the solution of 
\bq
 \lt\{\begin{array}{rcl}
 L_n \phi_n &=& -2 (u_{n,1}')^2\\
    \phi_n (L) &=& 0 \\
\end{array}\rt.\eq
Again by It\^o formula, we get 
$$\phi_n(R({\tau_n})) - \phi_n(0) = - 2  \int_0^{\tau_n} (u_{n,1}')^2(R(s))ds + M_{\tau_n} ,$$
where $(M_t)_{t\in[0,\tau_n]}$ is a  martingale.
After taking the expectation in the above formula, we get from  Proposition \ref{green}:
$$\Var(\tau_n)= \phi_n(0)  = 2\int_0^{L} \frac{f^{n-1}(t)}{I_n^2(t)} \lt(\int_0^t \frac{I_n^2(s)} {f^{n-1}(s)} (u_{n,1}'(s))^2 ds\rt) dt$$
\wwtbp

\section{Application to  cut-off for rotationnaly symmetric }
In this section we  derive the cut-off in separation phenomenon for a class of rotationally symmetric manifolds that contains  spheres. \par
From now on, all constants will be denoted $c$, their exact values can change from one line to another. When these constants depend on a parameter, such as $A$, we will rather write $c(A)$. 

\begin{pro}\label{mean-ge}
Let $f $ be a $C^3$ function on $[0, L]$ satisfying Assumptions \eqref{H_f} and  \eqref{H_f2}. Assume that $f''(L/2) <0 $. Denote  for $n\ge 1$ $\displaystyle a_n:=\frac{f(L/2)}{\vert f''(L/2)\vert}\frac{\ln(n)}{n}$, and let $(b_n)_{n\ge 1}$ satisfy $\displaystyle \frac{\sqrt{\ln n}}{n}=o(b_n)$ and $b_n\le a_n$. Then
\begin{equation}
\label{E0} 
\EE[ \tau_n] - a_n=o(b_n).
\end{equation}
\end{pro}
\proof

From Proposition \ref{mean-var} and Remark \ref{rem-int} we get:
$\EE[\tau_n] = u_{n,1}(0)$.
Let us write, for $\displaystyle A_n:=e^{\sqrt{\ln n}}$:

\begin{align}\label{u_ge}
u_{n,1}(0)&=  \int_0^{L}  \frac{I_n(s)}{f^{n-1}(s)} \lt( \frac{I_n(L)-I_n(s)}{I_n(L)}\rt)  ds = 2\int_0^{L/2} \frac{I_n(s)}{f^{n-1}(s)}\lt( \frac{I_n(L)-I_n(s)}{I_n(L)}\rt)  ds \notag \\
&=2\left( \alpha_n+\gamma_n+\beta_n\right)
\end{align}
with 
\begin{equation}
\label{E1}
\alpha_n:= \int_0^{L/2 - A_n/\sqrt{n}}  \frac{I_n(s)}{f^{n-1}(s)} \lt( \frac{I_n(L)-I_n(s)}{I_n(L)}\rt)  ds,
\end{equation}
\begin{equation}
\label{E2}
\gamma_n:=\int_{L/2 - A_n/\sqrt{n}}^{L/2- 1/\sqrt{n} }   \frac{I_n(s)}{f^{n-1}(s)} \lt( \frac{I_n(L)-I_n(s)}{I_n(L)}\rt)  ds
\end{equation}
and 
\begin{equation}
\label{E3}
\beta_n:=
\int_{L/2 - 1/\sqrt{n}}^{L/2 }   \frac{I_n(s)}{f^{n-1}(s)} \lt( \frac{I_n(L)-I_n(s)}{I_n(L)}\rt)  ds.
\end{equation}
We will prove that 
\begin{equation}
\label{E4}
a_n-2\alpha_n=o(b_n),\quad \gamma_n=o(b_n)\quad\hbox{and}\quad \beta_n=o(b_n)
\end{equation}
and this will establish~\eqref{E0}.

\medbreak

Let us prove that $\displaystyle a_n-2\alpha_n=o(b_n)$.

 We introduce  \bqn{J_def_ge}
  J_n(s) &\df& \frac{I_n(s)}{f^{n-1}(s)} \frac{I_n(L)-I_n(s)}{I_n(L)}
=  \frac{I_n(s)}{f^{n-1}(s)} \frac{I_n(L-s)}{I_n(L)}.
\eqn

We have 
\bqn{E8}
\alpha_n:= \int_0^{L/2 - A_n/\sqrt{n}}  J_n(s) ds
\eqn

We will prove that 
\begin{equation}
\label{E6}
\forall\ s\in \left[0,L/2 - A_n/\sqrt{n}\right],\quad\left| \frac{f(s)}{n f'(s)}-J_n(s) \right|=O\left(\frac{f(s)}{n A_n^2f'(s)}\right)
\end{equation}
uniformly in $s$,
and then that 
\begin{equation}
\label{E7}
\frac1n\int_0^{L/2 - A_n/\sqrt{n}}\frac{f(s)}{ f'(s)}\, ds-\frac{a_n}{2}+\frac{\sqrt{\ln n}}{n}=O(1/n).
\end{equation}
From estimates~\eqref{E6} and~\eqref{E7} we will get  $\displaystyle a_n-2\alpha_n=o(b_n)$.

For $ s= \frac{L}{2}$, since $ f$ is increasing in $[0,L/2]$
and $ f'(L/2) = 0$ we get using Laplace's method:
\begin{align}\label{I_L_ge}
 I_n\lt(\frac{L}{2}\rt) &= \int_0^{\frac{L}{2}} f^{n-1}(t)dt = \int_0^{\frac{L}{2}} \exp((n-1) \ln (f)(t) )dt \notag \\
 & \sim_{n \ri \infty}  \sqrt{\frac{\pi }{2\vert \ln (f(s))''\vert_{s= L /2} \vert}} \frac{f^{n-1}(L/2)}{\sqrt{n-1}} \notag \\
 &\sim_{n \ri \infty}  \sqrt{  \frac{\pi f(L/2)}{2n \vert f''(L/2)\vert}} f^{n-1}(L/2).
 \end{align}
 
 If $s <\frac{L}{2}-\frac{A_n}{\sqrt n}  $ then 
%using again Laplace method 
we have the following expansion, since $f'>0$  on the interval $[0, L/2)$:
\begin{align}\label{I_ge}
 I_n(s) &= \int_0^s f^{n-1}(t)dt = \int_0^s \exp((n-1) \ln (f(t)) )dt \notag  \\
 & \le \frac{L}{2} \exp\left((n-1) \ln \left(f\left(\frac{L}{2}-\frac{A_n}{\sqrt n} \right)\right) \right)\notag\\
&= \frac{L}{2}f\left(\frac{L}2\right)^{n-1}\exp\left[(n-1)\ln  \left( 1+\frac{A_n^2f''(L/2)}{2nf(L/2)}+O\left(\frac{A_n^3}{n^{3/2}}\right)\right)\right]\notag\\
&= \frac{L}{2}f\left(\frac{L}2\right)^{n-1}\exp\left[\frac{A_n^2f''(L/2)}{2f(L/2)}+O\left(\frac{A_n^3}{n^{1/2}}\right)\right].
 \end{align}

This implies that 
\begin{align*}
\left|\frac{I_n(L)-I_n(s)}{I_n(L)}-1\right|=\frac{L}{2}\frac{I_n(s)}{I_n(L)}&\le 
\frac{L}{2}\frac{f\left(\frac{L}2\right)^{n-1}}{I_n(L)}\exp\left[\frac{A_n^2f''(L/2)}{2f(L/2)}+O\left(\frac{A_n^3}{n^{1/2}}\right)\right]\\
&\sim \frac{L}4\sqrt{\frac{2n|f''(L/2)|}{\pi f(L/2)}}\exp\left[\frac{A_n^2f''(L/2)}{2f(L/2)}+O\left(\frac{A_n^3}{n^{1/2}}\right)\right]\\
&\le \frac{C}{n}
\end{align*}
for some $C>0$ (this majoration will be enough for our purpose). We get
\begin{equation}
\label{E5}
1-\frac{C}{n} \le   \frac{I_n(L)-I_n(s)}{I_n(L)}\le 1.
\end{equation}
  %Hence for  $s$ fixed, we have  $  I_n(s) = o( I_n(L/2)) = o( I_n(L))$ and so $I_n(L) - I_n(s)  \sim I_n(L)$.
  
%  So for fixed $s\in[0,L/2)$, we have $\frac{I_n(s)}{f^{n-1}(s)} \lt( \frac{I_n(L)-I_n(s)}{I_n(L)}\rt) \sim \frac{f(s)}{nf'(s)}$.
 
%In the following computations, we will show that the equivalent \eqref{I_ge} is in fact uniform in a large interval. 
We now investigate the term $I_n(s)/f^{n-1}(s)$ of $J_n(s)$.
After integration by parts, for $ 0\le s < L/2$, and since $f'>0 $ on $[0,L/2)$, we have
\begin{align}\label{I_cal_ge}
 I_n(s) &= \int_0^s f^{n-1}(t)dt = \int_0^s \frac{f'(t)f^{n-1}(t)}{f'(t)}dt \notag \\
 &=\frac{f^{n}(s)}{nf'(s)} + \int_0^s \frac{f^{n}(t)f''(t)}{n (f'(t))^2}dt.
\end{align}
\par
Since $f'$ is decreasing and positive on $[0,L/2)$, we have for $s\in[0,L/2)$,
\bq
\fo t\in[0,s],\qquad \frac{f^{n}(t)}{(f'(t))^2}& \le& \frac{f^{n}(t)}{(f'(s))^2} ,\eq 
hence, with $ m \df \min_{[0,L/2]} f'' <0$,
\begin{equation*}
 \frac{f^n(s)}{nf'(s)} + \frac{m}{n(f'(s))^2} I_{n+1}(s)   \le I_n(s)\le \frac{f^n(s)}{nf'(s)}.
\end{equation*}
From the above equation we get
\begin{equation}\label{I_born}
 \frac{f^n(s)}{nf'(s)}  \lt(1 + \frac{m f(s)}{(n+1)(f'(s))^2}\rt)   \le I_n(s)\le \frac{f^n(s)}{nf'(s)}.
\end{equation}
Since $f$ is increasing, $f'$ is non-increasing in $(0,L/2)$ and $m<0$ we deduce that for $s\in [0,L/2 - A_n/\sqrt{n}]$:

\begin{equation}
\label{E81}
 \frac{f^n(s)}{nf'(s)}  \lt(1 - \frac{|m| f(L/2)}{(n+1)(f'(L/2 - A_n/\sqrt{n}))^2}\rt)   \le I_n(s)\le \frac{f^n(s)}{nf'(s)}.
\end{equation}
From
\begin{equation}
\label{E13}
f'\left(L/2 - A_n/\sqrt{n}\right)+\frac{A_n}{\sqrt n}f''(L/2)=O\left(A_n^2/n\right)
\end{equation}
we get 
\[
f'\left(L/2 - A_n/\sqrt{n}\right)^2-\frac{A_n^2}{ n}f''(L/2)^2=O\left(\frac{A_n^3}{n^{3/2}}\right)
\]
which yields
\[
\frac{|m| f(L/2)}{(n+1)(f'(L/2 - A_n/\sqrt{n}))^2}=O\left(\frac1{A_n^2}\right).
\]
This estimate together with~\eqref{E81} give for $s\in [0,L/2 - A_n/\sqrt{n}]$:
\bqn{E9}
0\le \frac{f(s)}{nf'(s)}-\frac{I_n(s)}{f^{n-1}(s)}\le \frac{C}{A_n^2}\frac{f(s)}{nf'(s)}
\eqn
for some $C>0$. Multiplying by $\displaystyle  \frac{I_n(L)-I_n(s)}{I_n(L)}$ and using~\eqref{E5} we get for all $s\in [0,L/2 - A_n/\sqrt{n}]$,
\bqn{10}
\left|\frac{f(s)}{nf'(s)}-J_n(s)\right|\le \frac{Cf(s)}{A_n^2 n f'(s)}
\eqn
for some $C>0$. This is~\eqref{E6}.

For proving~\eqref{E7} we remark that a Taylor expansion of $\displaystyle \frac{f(s)}{f'(s)}$ yields
on $s\in [0,L/2 )$
\begin{equation}\label{E10}
\frac{f(s)}{f'(s)}=\frac{f(L/2)}{|f''(L/2)|(L/2-s)}+g(s)
\end{equation}
with $g(s)$ uniformly bounded in  $[0,L/2 )$. So
\begin{align*}
\frac1n\int_0^{ L/2 - A_n/\sqrt{n}}\frac{f(s)}{f'(s)}\, ds&=\frac{f(L/2)}{|f''(L/2)|}\left(\frac{\ln n}{2n}-\frac{\ln A_n}n\right)+O(1/n)\\
&= \frac{a_n}2-\frac{\sqrt{\ln n}}n+O(1/n)
\end{align*}
which is~\eqref{E7}. So $a_n-2\alpha_n=o(b_n)$.

\medbreak

Next we prove that $\gamma_n=o(b_n)$. We already know as an immediate consequence of~\eqref{I_born} that on $[0, L/2-1/\sqrt{n}]$, we have 
\begin{equation}
\label{E12}
 J_n(s)\le \frac{f(s)}{nf'(s)}.
\end{equation}
 On the other hand, by~\eqref{E10}, 
$\displaystyle
\frac{f(s)}{f'(s)}-\frac{f(L/2)}{|f''(L/2)|(L/2-s)}$ is bounded in $[0, L/2)$ .Consequently,
\begin{align}\label{E11}
\frac1n\int_{L/2-A_n/\sqrt{n}}^{L/2-1/\sqrt{n}}\frac{f(s)}{f'(s)}\, ds &\sim \frac1n\int_{L/2-A_n/\sqrt{n}}^{L/2-1/\sqrt{n}}\frac{f(L/2)}{|f''(L/2)|(L/2-s)}\, ds\notag\\
&=\frac1n\frac{f(L/2)}{|f''(L/2)|}\ln(A_n)=\frac1n\frac{f(L/2)}{|f''(L/2)|}\sqrt{\ln n}=o(b_n).
\end{align}
This proves the second estimate in~\eqref{E4}.  

\medbreak

Finally we prove that $\beta_n=o(b_n)$.

 If $s \in [\frac{L}{2}-\frac{1}{\sqrt{n}}, \frac{L}{2}+\frac{1}{\sqrt{n}}]$, then write $s = L/2 + a/\sqrt{n} $, with $a \in [-1,1] $. Since $f'(L/2) = 0 $ and $f$ is $C^3$, we have uniformly in $a \in [-1,1] $:
\begin{align}
I_n(L/2 + a/\sqrt{n} ) &= I_n(L/2) + \int_{L/2}^{L/2 + a/\sqrt{n}} f^{n-1}(x)dx \notag \\
&= I_n(L/2) + \frac{1}{\sqrt{n}} \int_{0}^{a} f^{n-1}\lt(\f{L}2 + \frac{h}{\sqrt{n}}\rt)dh \notag \\
&= I_n(L/2) + \frac{1}{\sqrt{n}}\int_{0}^{a} \lt(f(L/2) + \frac{f''(L/2)h^2}{2n} +O(1/n^{3/2}) \rt)^{n-1}dh \notag \\
&\sim \frac{ f^{n-1}(L/2)}{\sqrt{n}} \lt(  \sqrt{  \frac{\pi f(L/2)}{2 \vert f''(L/2)\vert}} + \int_0^a e^{ \frac{f''(L/2)h^2}{2f(L/2)} } dh   \rt) \notag \\
&= \frac{ f^{n-1}(L/2)}{\sqrt{n}}   \int_{-\infty }^a e^{ \frac{f''(L/2)h^2}{2f(L/2)} } dh  .\label{lemme}
\end{align} 
 Hence, letting $ h(a) = \int_{-\infty }^a e^{ \frac{f''(L/2)h^2}{2f(L/2)} } dh $,  we get that  for $\beta_n$ defined in \eqref{u_ge}
 \bqn{B_n_g}
 \beta_n &\df &\int_{L/2 - 1/\sqrt{n}}^{L/2 }   \frac{I_n(s)}{f^{n-1}(s)} \lt( \frac{I_n(L)-I_n(s)}{I_n(L)}\rt)  ds \notag \\
 &\sim &\frac{1}{\sqrt{n}I_n(L)} \int_{-1}^{0} \frac{h(a)h(-a) f^{2n-2}(L/2)}  {nf^{n-1}(L/2)  e^{\frac{f''(L/2)a^2}{2f(L/2)}}  } da \notag \\
 &\sim &\frac{c}{n}\ =\ o(b_n) ,
 \eqn

where we used the following uniform estimate in $a\in[-1,1]$ obtained as in~\eqref{I_ge}
\begin{equation}
\label{E14}
f^n(L/2-a/\sqrt{n})\sim f^n(L/2) e^{\frac{f''(L/2)a^2}{2f(L/2)}},
\end{equation}
 and next \eqref{I_L_ge} for the last equivalent.

\wwtbp

\begin{pro}\label{var-ge}
Let $f $ be a $C^3$ function on $[0, L]$ satisfying Assumptions \eqref{H_f} and  \eqref{H_f2}. Assume that $f''(L/2) <0 $, then
$$ \Var(\tau_n)  = O\left( \frac{\ln n}{n^2}\right).$$
\end{pro}

\proof
From Proposition \ref{var} and after integration by parts, it follows that  for all $A$ large enough and for all $ n $ large enough:
\begin{align}
 \frac{\Var(\tau_n)}{2} &=  \int_0^{L} \frac{f^{n-1}(t)}{I_n^2(t)} \int_0^t \frac{I_n^2(s)} {f^{n-1}(s)} (u_{n,1}'(s))^2 ds \notag \\
 &=  \lt[ -\frac{1}{I_n(t)} \int_0^t \frac{I_n^2(s)}{f^{n-1}(s)}(u_{n,1}'(s))^2  ds \rt]_0^{L}\notag + \int_0^{L}\frac{I_n(s)}{f^{n-1}(s)}(u_{n,1}'(s))^2 ds\notag \\
 &= \int_0^{L}\frac{I_n(s)}{f^{n-1}(s)} \frac{I_n(L)-I_n(s)}{I_n(L)}(u_{n,1}'(s))^2 ds \notag \\
 &= \underbrace{\int_0^{L/2- A/\sqrt{n}} J_n(s) (u_{n,1}'(s))^2 ds}_{A_n}
 + \underbrace{\int_{L/2- A/\sqrt{n}}^{L/2+ A/\sqrt{n}}J_n(s) (u_{n,1}'(s))^2 ds}_{B_n}\\
& + \underbrace{\int_{L/2+ A/\sqrt{n}}^{L} J_n(s) (u_{n,1}'(s))^2 ds}_{C_n} \label{**_ge_k1} \\
 \notag
\end{align} 
We will analyze the three last terms separately.
Recall that by Proposition \ref{Poisson} and from the fact that $u_{n,1}=G_n[\un]$, we have 
\bqn{dv}(u_{n,1}'(t))^2 &= &\lt( \frac{f^{n-1}(t)}{I^2_n(t)} \int_0^t \frac{I^2_n(s)}{f^{n-1}(s)} ds \rt)^2\eqn
\begin{itemize}

\item Let us start by the term $ A_n$, using computations \eqref{I_cal_ge}  and \eqref{I_born},
since $f$ is increasing, $f'$ is non-increasing in $(0,L/2)$ and $m<0$,
 it follows that for  $ A$ big enough and for all $n$ sufficiently large, and for all $s \in [0, L/2 -A/\sqrt{n} ], $ 
\begin{equation}\label{21_ge}
\frac{f^{n+1}(s)}{n^2 (f'(s))^2} \lt(1-\frac1A\rt)^2 \le \frac{I^2_n(s)}{f^{n-1}(s)} \le \frac{f^{n+1}(s)}{n^2 (f'(s))^2}.
 \end{equation}
Let $W_n(t) = \int_0^t \frac{f^{n+1}(s)}{(f'(s))^2}ds $, for $ 0 \le  t < L/2  $, we have after integration by parts:
\begin{align*}
W_n(t) &= \int_0^t \frac{f^{n+1}(s)f'(s)}{(f'(s))^3}ds \\
&= \frac{f^{n+2}(t)}{(n+2) (f'(t))^3}+\frac{3}{n+2}  \int_0^t \frac{f^{n+2}(s)f''(s)}{(f'(s))^4}ds .\\
\end{align*}
Since $ f'' \le 0$, we deduce, using \eqref{21_ge},
\begin{equation}\label{Int_I/f^}
 \int_0^t \frac{I^2_n(s)}{f^{n-1}(s)} ds   \le  \frac{f^{n+2}(t)}{n^3 (f'(t))^3}
 \end{equation}
and that  for  $ A$ big enough and for all $n$ sufficiently large:
\bq
(u_{n,1}'(t))^2 &\le&
 \lt(\frac{n^2(f'(t))^2f^{n+2}(t)}{f^{n+1}(t) (1-\frac1A)^2n^3 (f'(t))^3}  \rt)^2 
\\
&=& \lt(\frac{f(t)}{nf'(t) (1-\frac1A)^2 }  \rt)^2 .
\eq
Also by   \eqref{E12},  
$$ J_n(s) \df \frac{I_n(s)}{f^{n-1}(s)} \frac{I_n(L-s)}{I_n(L)} \le \frac{f(s)}{n f'(s)}. $$

 Hence for $A$ big enough, for all $n$ sufficiently large, and for $A_n$ defined in \eqref{**_ge_k1}, by \eqref{21_ge},  we have
\begin{align*}
A_n &\le \frac{1}{n^3 (1 - \frac{1}{A})^4} \int_0^{L/2- A/\sqrt{n}}\frac{f^{3}(s)}{(f'(s))^3} ds\\
& \le \frac{f(L/2)^3}{n^3 (1 - \frac{1}{A})^4} \int_0^{L/2- A/\sqrt{n}}\frac{1}{(f'(s))^3} ds\\
&\sim \frac{c}{A^{2} n^{2}},
\end{align*}
where in the last equivalent we used $ \frac{1}{f'(s)} \sim_{s \ri L/2-} \frac{1}{\vert f''(L/2)\vert (L/2-s)} . $

Hence  for $A $ big enough,  \bqn{compl2} A_n &=& O\lt(\frac{\ln(n)}{n^2}\rt).\eqn
\item For the term $B_n $ in \eqref{**_ge_k1}: for $A$ big enough, for all $n$  large enough and  for $ a \in [-A,A] $,  we have
$$\lt\vert u_{n,1}'\lt(\frac{L}{2} +\frac{a}{\sqrt{n} }\rt)\rt\vert =  \frac{f^{n-1}\lt(\frac{L}{2} +\frac{a}{\sqrt{n} }\rt)}{I^2_n\lt(\frac{L}{2} +\frac{a}{\sqrt{n} }\rt)} \lt(\int_0^{L/2-A/\sqrt{n}} \frac{I^2_n(s)}{f^{n-1}(s)} ds  +  \int_{L/2-A/\sqrt{n}}^{L/2+\frac{a}{\sqrt{n} }} \frac{I^2_n(s)}{f^{n-1}(s)} ds  \rt) . $$

Let $C(f,2) =   \lt(  \frac{ 2f(L/2)}{ \vert f''(L/2)\vert} \rt)^{1/2} $ we have (similarly to the computation in~\eqref{I_ge} and~\eqref{E13}):
\bqn{compl3} f^n(L/2 - A/n^{1/2} )&\sim& f^n(L/2) e^{\frac{-A^{2}}{C(f,2)^{2}}},\eqn
and  \bqn{compl4}f'(L/2 - A/n^{1/2}) &\sim& \frac{ \vert f''(L/2) \vert A}{ n^{1/2}}. \eqn

By the above computation and \eqref{Int_I/f^}, we have 

 $$ \int_0^{L/2-A/\sqrt{n}} \frac{I^2_n(s)}{f^{n-1}(s)} ds \le \frac{f^{n+2}(L/2-\frac{A}{\sqrt{n} })}{n^3 (f'(L/2-\frac{A}{\sqrt{n} }))^3} \sim \frac{f^{n+2}(L/2)e^{-\frac{{A^2}}{{C(f,2)^{2}}}}}{n^{3/2}A^3\vert f''(L/2) \vert^3}.$$

Recall that from \eqref{lemme}, and for $h_1 (a) = \int_{-\infty }^a  e^{ \frac{-h^{2}}{C(f,{2})^{2}}}dh $, we have uniformly over  $ a \in [-A,A] $
   \begin{equation} I_n(L/2 + a/\sqrt{n}) \sim f^{n-1}(L/2)\f{h_1(a)}{\sqrt{n}}, \label{10_gen_k1} 
   \end{equation}
   
  hence since uniformly in $ a \in [-A,A] $,
 
 $$ f^n(L/2 - a/n^{1/2} )\sim f^n(L/2) e^{\frac{-a^{2}}{C(f,2)^{2}}},$$
  taking \eqref{compl3} into account,  
\bqn{C20}
\nonumber \int_{L/2-A/\sqrt{n}}^{L/2+\frac{a}{\sqrt{n} }} \frac{I^2_n(s)}{f^{n-1}(s)} ds  &=&  \frac{1}{\sqrt{n}}\int_{-A}^{a} \frac{I^2_n(L/2+\frac{\tilde{a}}{\sqrt{n} })}{f^{n-1}(L/2+\frac{\tilde{a}}{\sqrt{n} })} d \tilde{a}\\
 & \sim & \frac{f^{n-1}(L/2)}{\sqrt{n}} \int_{-A}^{a} \frac{h_1^2(\tilde{a})e^{\frac{\tilde{a}^2}{C(f,2)^{2}}}}{n} d \tilde{a} =\frac{f^{n-1}(L/2)\theta_{A}(a)}{n^{3/2}},
\eqn
where $\theta_{A}(a) \df \int_{-A}^{a} h_1^2(\tilde{a})e^{\frac{\tilde{a}^2}{C(f,2)^{2}}} d \tilde{a} $.

 We have  for $A$ big enough  
 $$\vert u_{n,1}'\lt(\frac{L}{2} +\frac{a}{\sqrt{n} }\rt) \vert \le c \frac{e^{-{\frac{a^2}{C(f,2)^{2}}}}} {\sqrt{n}h_1^2(a)} \lt(\theta_{A}(a) + \frac{1}{A^{3}}\rt).$$
Hence 
\bq
 B_n &\df& \int_{L/2- A/\sqrt{n}}^{L/2+ A/\sqrt{n}}J_n(s) (u_{n,1}'(s))^2 ds\\
 &= &\frac{1}{\sqrt{n}}\int_{-A}^{A}J_n(L/2+ a/\sqrt{n}) (u_{n,1}'(L/2+ a/\sqrt{n}))^2 da\\
 &\le &\frac{1}{\sqrt{n}}\int_{-A}^{A}J_n(L/2+ a/\sqrt{n}) \lt(c \frac{e^{-{\frac{a^2}{C(f,2)^{2}}}}} {\sqrt{n}h_1^2(a)} \lt(\theta_{A}(a) + \frac{1}{A^{3}}\rt)\rt)^2 da\\
 &\sim&  \frac{c(A)}{n^2 },\\
\eq
where, taking \eqref{10_gen_k1} into account,  we use for the last term that
$$ J_n(L/2 + a/\sqrt{n})=\frac{I_n(L/2 +a/\sqrt{n})}{f^{n-1}(L/2 +a/\sqrt{n})}  \frac{I_n(L/2 -a/\sqrt{n})}{I_n(L)} \sim c \frac{h_1(a)h_1(-a)e^{\frac{a^2}{C(f,2)^{2}}}}{\sqrt{n} } ,$$ and $c(A) $ is a constant that depends on $A$.
It follows that once $A$ is fixed, for $n$ big enough,  
\bqn{compl5} B_n &=&   O\lt(\frac{\ln(n)}{n^2}\rt).\eqn

\item For the last term $C_n $ in \eqref{**_ge_k1}, note that $J_n(s)=J_n(L -s)$ so
\begin{align*}
C_n &= \int_{L/2+ A/\sqrt{n}}^{L} J_n(s) (u_{n,1}'(s))^2 ds\\
&=\int_0^{L/2- A/\sqrt{n}} J_n(s) (u_{n,1}'(L-s))^2 ds.\\
\end{align*}
Also for $ s \le L/2- A/\sqrt{n} $
\begin{align*}
&\vert u_{n,1}'(L - s) \vert =  \frac{f^{n-1}(s)}{I^2_n(L-s)} \int_0^{L-s} \frac{I^2_n(t)}{f^{n-1}(t)} dt \\
&= \frac{f^{n-1}(s)}{I^2_n(L-s)} \lt( \int_0^{L/2-A/\sqrt{n}} \frac{I^2_n(t)}{f^{n-1}(t)} dt   +  \int_{L/2-A/\sqrt{n}}^{L/2+A/\sqrt{n}} \frac{I^2_n(t)}{f^{n-1}(t)} dt 
% \rt.\\ & \quad \lt.
 + \int_{L/2+A/\sqrt{n}}^{L-s} \frac{I^2_n(t)}{f^{n-1}(t)} dt \rt).
\end{align*}
Using \eqref{Int_I/f^}, \eqref{compl3} and \eqref{compl4}, we have that for $A$ big enough, for all $n$ sufficiently large 
$$ \int_0^{L/2-A/\sqrt{n}} \frac{I^2_n(s)}{f^{n-1}(s)} ds \le 2  \frac{f^{n+2}(L/2)e^{-\frac{{A^2}}{{C(f,2)^{2}}}}}{n^{3/2}A^3(f''(L/2))^3} $$
and using \eqref{C20},
$$ \int_{L/2-A/\sqrt{n}}^{L/2+A/\sqrt{n}} \frac{I^2_n(t)}{f^{n-1}(t)} dt \le 2  \frac{\theta_{A}(A)f^{n-1}(L/2)}{n^{3/2}}.$$
For the last term since for $s\leq L/2-A/\sqrt{n}$,
\begin{align*}
\int_{L/2+A/\sqrt{n}}^{L-s} \frac{I^2_n(t)}{f^{n-1}(t)} dt
&\le I^2_n(L-s)\int_{L/2+A/\sqrt{n}}^{L-s} \frac{1}{f^{n-1}(t)} dt\\
&=  I^2_n(L-s)\int_s^{L/2-A/\sqrt{n}}\frac{1}{f^{n-1}(t)} dt\\
&=  I^2_n(L-s)\int_s^{L/2-A/\sqrt{n}}\frac{f'(t)}{f^{n-1}(t)f'(t)} dt\\
&\le \frac{I^2_n(L-s)}{f'(L/2-A/\sqrt{n})} \lt(\frac{f^{-n +2}(s)}{n-2} \rt) \\
&\sim c\frac{f^{-n+2}(s)I^2_n(L-s)}{n^{1/2}A}\\
\end{align*}
Since $L -s \ge L/2 $, we have for $A$ big enough and for all $n$ sufficiently large,
\begin{align*}
\vert u_{n,1}'(L - s)\vert &\le  c\frac{f^{2n-2}(L/2)}{n^{3/2}I_n^2(L/2)}\lt(  \frac{e^{-\frac{{A^2}}{{C(f,2)^{2}}}}}{A^3} + \theta_{A}(A)   \rt) +  c \frac{f(s)}{n^{1/2}A}\\
 &\le   \frac{c}{n^{1/2}} \lt(  \frac{e^{-\frac{{A^2}}{{C(f,2)^{2}}}}}{A^3} + \theta_{A}(A)   \rt)
 +c \frac{f(s)}{n^{1/2}A}
 \le  c(A) \frac{1}{n^{1/2}} 
\end{align*}
where in the second inequality, we used \eqref{I_L_ge}. Since, for $ s \le L/2 - A/\sqrt{n} $, by \eqref{I_born}  $ J_n(s) \le \frac{f(s)}{n f'(s)} $, for $A$ big enough and for all $n$ sufficiently large,
\begin{align*}
C_n 
&=\int_0^{L/2- A/\sqrt{n}} J_n(s) (u_{n,1}'(L-s))^2 ds.\\
&\le \frac{c(A)^2}{n} \int_0^{{L/2- A/\sqrt{n}} } J_n(s)ds\\
&\le  \frac{c(A)^2 f(L/2)}{n^2} \int_0^{{L/2- A/\sqrt{n}} } \frac{1}{f'(s)} ds\\
&\sim  \frac{c(A)^2 f(L/2)}{n^2}  \lt(\frac12 \ln(n) \rt)  \sim c(A)^2 \frac{\ln(n)}{n^2}
\end{align*}
Hence \bqn{compl6}C _n &= &  O\lt(\frac{\ln(n)}{n^2}\rt).\eqn
\end{itemize}
Putting \eqref{compl2}, \eqref{compl5} and \eqref{compl6} together, we deduce that $\displaystyle \Var(\tau_n)= O\lt(\frac{\ln(n)}{n^2}\rt) .$
\wwtbp
\par
We deduce the following result

\begin{theo}\label{cut-off-ge}
Let $f $ be a $C^3$ function on $[0, L]$ satisfying Assumptions \eqref{H_f} and  \eqref{H_f2} and $f''(L/2) <0 $.
For $n\in\NN\setminus\{1\}$,  consider  the Brownian motion $X_n\df(X_n(t))_{t\geq 0}$ described in Definition \ref{BM}.
Then the family of diffusion processes $(X_n)_{n\in\NN\setminus\{1\}}$ has a cut-off in separation in the sense of Section \ref{cutoff}, with mixing times $(a_n)_{n\in\NN\setminus\{1\}} = \lt(\frac{f(L/2)}{\vert f''(L/2)\vert}\frac{\ln(n)}{n}\rt)_{n\in\NN\setminus\{1\}}$ and windows $(b_n)_{n\in\NN\setminus\{1\}}$ satisfying $\forall \ n\ge 1$, $b_n\le a_n$,  together with $\displaystyle \frac{\sqrt{\ln(n)}}{n}=o(b_n)$. 
\end{theo}

\proof
Use Theorem \ref{th1},  Proposition \ref{mean-ge} and Proposition \ref{var-ge}.
\wwtbp
\par
The previous arguments provide an alternative proof to the main result from \cite{sphere}, and strengthens the result with a cut-off in separation with windows:
\begin{cor}\label{sphere}
Let  $X_n\df(X_n(t))_{t\geq 0}$ be the Brownian motion described in Definition \ref{BM}, where $M_f^n$
is replaced by
 the sphere $ \mathbb{S}^n$ and where $\wi 0$ now stands for any point of $ \mathbb{S}^n$.
Then the family of diffusion processes $(X_n)_{n\in\NN\setminus\{1\}}$ has a cut-off in separation in the sense of Section \ref{cutoff},  with mixing times $(a_n)_{n\in\NN\setminus\{1\}} =\lt( \frac{\ln(n)}{n}\rt)_{n\in\NN\setminus\{1\}}$ and windows $(b_n)_{n\in\NN\setminus\{1\}}$ satisfying $\forall \ n\ge 1$, $b_n\le a_n$,  together with $\displaystyle \frac{\sqrt{\ln(n)}}{n}=o(b_n)$. 
\end{cor}
\proof
Use Theorem \ref{cut-off-ge}, with $f = \sin $ and $L=\pi$, and note that by symmetry  in this case the starting point is not relevant.
\wwtbp
\par
%The symmetry in the above argument is only valid for the sphere, see the appendix.
%%%%%%%%%%%%%%%%%%%%%%% rajout P_t %%%%%%%%%%%%
We  also deduce the following consequences. %%%%%%%\ref{} %%%%%%%%
\begin{cor}\label{cor_ker}
Let $f $ be a $C^3$ function on $[0, L]$ satisfying Assumptions \eqref{H_f} and  \eqref{H_f2} and $f''(L/2) <0 $.
For $n\in\NN\setminus\{1\}$,  consider  the Brownian motion $X_n\df(X_n(t))_{t\geq 0}$ in $M_f^n$ described in Definition \ref{BM}.
 There exist $C >0$ and $n_0 \in \mathbb{N} $ such that for all  $ r > 0$, $ 0<r' <1 $ and for all $ n \ge n_0$,
 \bq
   \lVe \mathcal{L}\lt(X_n({(1+r)\frac{f(L/2)}{\vert f''(L/2)\vert}\frac{\ln(n)}{n} }) \rt) - \mathcal{U}_n \rVe_{\mathrm{tv}} &\le &\frac{C}{r^2 \ln(n)}\\
\fo y \in M_f^n,\qquad    P^{(n)}_{(1+r)\frac{f(L/2)}{\vert f''(L/2)\vert}\frac{\ln(n)}{n}} ( \tilde{0},y)&\ge &\lt(1- \frac{C}{r^2 \ln(n)} \rt) \frac{1}{\Vol(M_f^n )} \\
\inf_{ y \in M_f^n}  P^{(n)}_{(1-r')\frac{f(L/2)}{\vert f''(L/2)\vert}\frac{\ln(n)}{n}} ( \tilde{0},y)&\le &\lt(\frac{C}{r'^2 \ln(n)} \rt) \frac{1}{\Vol(M_f^n )} \\
 \eq
where $\Vert \cdot  \Vert_{\mathrm{tv}}$ stands for the total variation norm,  $\mathcal{L}(X_n(t))$ is the law of $X_n(t)$, $\mathcal{U}_n $ is the uniform measure in $M_f^n$, and $P^{(n)}_t(\cdot,\cdot) $ is the heat kernel density at time $t>0$ associated to the Laplacian on $M_f^n $.
\end{cor}
 \proof
From the computations in the proof of Proposition \ref{mean-ge} and Proposition \ref{var-ge}, with $b_n = a_n $, there exist $C >0$ and $n_0 \in \mathbb{N} $ such that for all  $ r > 0$ and for all $ n \ge n_0$, 
\bq \PP\lt[\tau_n>(1+r)a_n \rt] \leq  \frac{\Var(\tau_n)}{((1+r) a_n - \EE[\tau_n] )^2}\leq \frac{C}{r^2 \ln(n)} \eq\par
 The first  conclusion follows, since 
\bq \lVe \mathcal{L}\lt(X_n({(1+r) a_n }) \rt) -  \mathcal{U}_n \rVe_{\mathrm{tv}}& \le &\fs\lt(\mathcal{L}\lt(X_n({(1+r)a_n })\rt),  \mathcal{U}_n \rt)\ \le\ \PP\lt[\tau_n>(1+r)a_n\rt]\eq
 The second conclusion follows by  definition of the separation discrepancy, since for all $y \in  M_f^n$ and $ t >0 $,
\bq 1 - P^{(n)}_t(\tilde{0},y) vol(M_f^n)  & \le& \fs(\mathcal{L}(X_n(t)),\mathcal{U}_n ) \eq
 The last conclusion follows in the same way.
 \wwtbp

%%%%%%%%%%%%%%%%%%%%%%% Fin rajout %%%%%%%%%%%%
\par
\begin{pro}\label{prop-mean}
Let $f $ be a $C^3$ function on $[0, L]$ satisfying Assumptions \eqref{H_f} and  \eqref{H_f2}. Assume that  for some $  k \ge 2$, $$f(L/2 + h ) = f(L/2) +  \frac{  f^{(2k)}(L/2)}{ (2k)!}  h^{2k} + o(h^{2k})$$ where $f^{(2k)}(L/2) < 0 $, then
$$\EE[ \tau_n] \sim  \frac{2kC({2k}) C(f,2k)^2 }{n^{1/k}\Gamma\lt(\frac{1}{{2k}}\rt)},  $$
%  \lt( \frac{ (2k)!f(L/2)}{ \vert f^{({2k})}(L/2)\vert} \rt)^{1/k}
 where $\Gamma $ is the usual Gamma functional and
 \bqn{C2}
 C(f,2k) &\df&   \lt(  \frac{ (2k)!f(L/2)}{ \vert f^{(2k)}(L/2)\vert} \rt)^{1/2k}\\
\label{C1} C({2k}) &\df& \int_0^{\infty}  h_{1,k}(a)h_{1,k}(-a)e^{a^{2k}} da, 
\eqn
and  for any $x\in\RR$, $ h_{1,k}(x) \df \int_{-\infty}^{x} e^{-a^{2k}} da  $.
\end{pro}
\proof
Recall from  Proposition \ref{mean-var}  that
$\EE[\tau_n] = u_{n,1}(0)=  \int_0^{L}  \frac{I_n(s)}{f^{n-1}(s)} \lt( \frac{I_n(L)-I_n(s)}{I_n(L)}\rt)  ds .$
Let us write, for $ A$ big enough and $n$ sufficiently large:

\begin{align}\label{u_ge_2}
&u_{n,1}(0)=  \int_0^{L}  \frac{I_n(s)}{f^{n-1}(s)} \lt( \frac{I_n(L)-I_n(s)}{I_n(L)}\rt)  ds = 2\int_0^{L/2} \frac{I_n(s)}{f^{n-1}(s)}\lt( \frac{I_n(L)-I_n(s)}{I_n(L)}\rt)  ds \notag \\
&=2\lt( \underbrace{\int_0^{L/2 - A/n^{1/2k}}  \frac{I_n(s)}{f^{n-1}(s)} \lt( \frac{I_n(L)-I_n(s)}{I_n(L)}\rt)  ds}_{A_n} \rt. \\
& \quad \quad \quad \quad \quad \quad + \lt. \underbrace{\int_{L/2 - A/n^{1/2k}}^{L/2 }   \frac{I_n(s)}{f^{n-1}(s)} \lt( \frac{I_n(L)-I_n(s)}{I_n(L)}\rt)  ds}_{B_n} \rt)
\end{align}

\begin{itemize}
\item Volume of $M_f^n$:
 Since $f$ is increasing in $[0, L/2]$, using   Laplace method we get:

\begin{align}\label{I_L_ge2}
 I_n\lt(\frac{L}{2}\rt) &= \int_0^{\frac{L}{2}} f^{n-1}(t)dt = \int_0^{\frac{L}{2}} \exp((n-1) \ln (f)(t) )dt \notag \\
 & \sim_{n \ri \infty} \frac{\Gamma\lt(\frac{1}{2k}\rt)}{2k} \lt(\frac{(2k)!}{ \vert (\ln f) ^{(2k)}(L /2)  \vert }\rt)^{1/2k} \frac{f^{n-1}(L/2)}{(n-1)^{1/2k}} \notag \\
 &\sim_{n \ri \infty}  \frac{1}{2k}\Gamma\lt(\frac{1}{2k}\rt)  \lt(  \frac{ (2k)!f(L/2)}{n \vert f^{(2k)}(L/2)\vert} \rt)^{1/2k} f^{n-1}(L/2) \notag \\
 &\sim_{n \ri \infty} \frac{1}{2k}\Gamma\lt(\frac{1}{2k}\rt) C(f,2k)  \frac{f^{n-1}(L/2)}{n^{1/2k}}
 \end{align}

Since $f''$ is non-positive, $f'(0)=1$, $f^{(i)}(L/2) = 0 $ for $ i \in \{1,...,2k-1\}$, and $f^{(2k)}(L/2) < 0 $, there exists $M > 0$ such that, for all $u\in [0,L]$
$$ \lt\vert  \frac{f''(u)}{(f'(u))^{(2k-2)/(2k-1)}} \rt\vert \le M .$$
It follows that for all $ s \in [0, L/2)$ and for all $t\in [0,s]$

$$  \lt\vert \frac{f^{n}(t)f''(t)}{n(f'(t))^2} \rt\vert \le \frac{f^{n}(t)M}{n(f'(t))^{2k/(2k-1)}} \le \frac{f^{n}(t)M}{n(f'(s))^{2k/(2k-1)}}.$$ 

Hence  \eqref{I_cal_ge} gives
\begin{equation}\label{I_born_ge_2}
 \frac{f^n(s)}{nf'(s)} -\frac{M}{n(f'(s))^{2k/(2k-1)}} I_{n+1}(s)   \le I_n(s)\le \frac{f^n(s)}{nf'(s)},
\end{equation}
it follows that
\begin{equation*}
 \frac{f^n(s)}{nf'(s)}  \lt(1 - \frac{Mf(s)}{(n+1)(f'(s))^{2k/(2k-1)}}\rt)   \le I_n(s)\le \frac{f^n(s)}{nf'(s)},
\end{equation*}
and
\begin{equation}\label{vol_born}
  1 - \frac{f^n(s)}{nf'(s)I_n(L)}  \le \frac{I_n(L)- I_n(s)}{I_n(L)}\le 1.
\end{equation}

\item For  $ s \in [0,L/2 - A/n^{1/2k}]$.\\
 The above equation gives:

\begin{equation*}
  1 - \frac{f^n(L/2 - A/n^{1/2k})}{nf'(L/2 - A/n^{1/2k})I_n(L)}   \le \frac{I_n(L)- I_n(s)}{I_n(L)}\le 1.
\end{equation*}
Since 
\bqn{C3} f^n(L/2 - A/n^{1/2k} )&\sim& f^n(L/2) e^{\frac{-A^{2k}}{C(f,2k)^{2k}}},\eqn
and  \bqn{E40}
f'(L/2 - A/n^{1/2k})& \sim&\frac{ \vert f^{(2k)}(L/2) \vert A^{2k-1}}{(2k-1)! n^{(2k-1)/2k}}, \eqn 
using \eqref{I_L_ge2}, we get for  $ A$ big enough and $n$ large enough 
$$  \frac{f^n(L/2 - A/n^{1/{2k}})}{nf'(L/2 - A/n^{1/{2k}})I_n(L)} \sim
\frac{   f(L/2) e^{\frac{-A^{2k}}{C(f,{2k})^{2k}}} }{  \frac{2}{(2k)!}\Gamma\lt(\frac{1}{{2k}}\rt) C(f,{2k})    \vert f^{({2k})}(L/2) \vert A^{2k-1}     } \le  e^{\frac{-A^{2k}}{C(f,{2k})^{2k}}} , $$
and so using the above equations, we get that for  $ A$ big enough  uniformly over $ s \in [0,L/2 - A/n^{1/{2k}}] $
\begin{equation}\label{vol_born_frac_2}
  (1 - e^{\frac{-A^{2k}}{C(f,{2k})^{2k}}}  )   \le \frac{I_n(L)- I_n(s)}{I_n(L)}\le 1.
\end{equation}
Since 
$$ \frac{Mf(s)}{(n+1)(f'(s))^{{2k}/(2k-1)}} \le \frac{Mf(L/2)}{(n+1)(f'(L/2 - A/n^{1/{2k}} ))^{{2k}/(2k-1)}} \sim \frac{ c} {A^{2k}}  $$
it follows with \eqref{E40}  that for $ A$ big enough and $n$ sufficiently large 
\begin{equation}\label{I_unif_gen}
 \frac{f^n(s)}{nf'(s)}  \lt(1 - \frac{1}{A^{2k-1}}\rt)   \le I_n(s)\le \frac{f^n(s)}{nf'(s)}.
\end{equation}
 Hence using \eqref{vol_born_frac_2} we get that for $ A$ big enough and for all  $n$ sufficiently large, uniformly in $s \in [0, L/2-A/n^{1/{2k}}]$:
 \begin{equation}\label{J_unif_2}
 \frac{f(s)}{nf'(s)}  \lt(1 - \frac{1}{A^{2k-1}}\rt) (1 -e^{\frac{-A^{2k}}{C(f,{2k})^{2k}}} )  \le J_n(s)\le \frac{f(s)}{nf'(s)}
 \end{equation}
 where $J_n$ is defined in \eqref{J_def_ge}.\par
 Since   $ \frac{f(s)}{f'(s)} \sim_{s \ri L/2-} \frac{(2k-1)!f(L/2)}{\vert f^{(2k)}(L/2)\vert (L/2-s)^{(2k-1)}} , $
we have 
 $$\frac1n \int_0^{L/2-A/n^{1/{2k}}} \frac{f(s)}{f'(s)}ds \sim \frac1n \int_0^{L/2-A/n^{1/{2k}}} \frac{(2k-1)!f(L/2)}{\vert f^{({2k})}(L/2)\vert (L/2-s)^{(2k-1)}}ds
 \sim \frac{c}{n^{1/k} A^{2k-2}} . $$

We get, for all $A$ big enough, and for $A_n$ defined in \eqref{u_ge_2}:
 \begin{equation} \label{sup_1_gen_2}
 \limsup_{n\ri\iy}  \frac{ A_n}{ \frac{c}{n^{1/k}}} = \limsup_{n\ri\iy}  \frac{ \int_{0}^{ L/2 -A/n^{1/{2k}} } J_n(s) ds}{ \frac{c}{n^{1/k}}} \le \frac{1}{A^{2k-2}}  
   \end{equation} 
   and 
    \begin{equation} \label{inf_1_gen_2} 
  \liminf_{n\ri\iy}   \frac{ A_n}{ \frac{c}{n^{1/k}}} =  \liminf_{n\ri\iy}   \frac{ \int_{0}^{ L/2 -A/ n^{1/{2k}}} J_n(s) ds}{ \frac{c}{n^{1/k}}}\ge   \frac{1}{A^{2k-2}} \lt(1 - \frac{1}{A^{2k-1}}\rt) (1 - e^{\frac{-A^{2k}}{C(f,{2k})^{2k}}} )  
  \end{equation}  
  
  \item If $s \in [\frac{L}{2}-\frac{A}{n^{1/{2k}} }, \frac{L}{2}+\frac{A}{n^{1/{2k}}} ]$, then write $s = L/2 + a/n^{1/{2k}}  $, with $a \in [-A,A] $. Since $f^{(i)}(L/2) = 0 $ for $ i \in \{1,...,2k-1\}$,  $f^{(2k)}(L/2) < 0 $ and $ f\in C^{2k+1}$, we deduce that uniformly in $a\in [-A,A] $:
\begin{align}
&I_n(L/2 + a/n^{1/{2k}}  ) = I_n(L/2) + \int_{L/2}^{L/2 + a/n^{1/{2k}} } f^{n-1}(x)dx \notag \\
&= I_n(L/2) + \frac{1}{n^{1/{2k}} } \int_{0}^{a} f^{n-1}\lt(\f{L}2 + \frac{h}{n^{1/{2k}} }\rt)dh \notag \\
&= I_n(L/2) + \frac{1}{n^{1/{2k}} }\int_{0}^{a} \lt(f(L/2) -  \frac{ \vert f^{({2k})}(L/2)\vert}{ {(2k)}!n}  h^{2k} + O(\frac{1}{n^{(2k+1)/{2k}}}) \rt)^{n-1}dh \notag \\
&\sim \frac{ f^{n-1}(L/2)}{n^{1/{2k}}} \lt( \frac{1}{{2k}}\Gamma\lt(\frac{1}{{2k}}\rt) C(f,{2k})   + \int_0^a e^{ \frac{-h^{2k}}{C(f,{2k})^{2k}}} dh   \rt)\notag  \\
&= \frac{ f^{n-1}(L/2)}{n^{1/{2k}}}   \int_{-\infty }^a e^{ \frac{-h^{2k}}{C(f,{2k})^{2k}}} dh  . \label{concentration}
\end{align} 
Let $ h_k(a) = \int_{-\infty }^a  e^{ \frac{-h^{2k}}{C(f,{2k})^{2k}}}dh $. Since
$  I_n(L)-I_n(s) = I_n(L-s)  $ we have uniformly in $ a \in [-A,A]$
$$ I_n(L/2 + a/n^{1/{2k}} ) ( I_n(L)-I_n( L/2 + a/n^{1/{2k}})  ) \sim \frac{ f^{2n-2}(L/2)}{n^{1/k}}h_k(a)h_k(-a).$$
Since 
\begin{align*}
\int_{L/2 - A/n^{1/{2k}}}^{L/2 }  & \frac{I_n(s)}{f^{n-1}(s)} \lt( \frac{I_n(L)-I_n(s)}{I_n(L)}\rt)  ds \\
&=  \frac{1}{n^{1/{2k}}I_n(L)} \int_{- A}^{0 } \frac{I_n(L/2 + a/n^{1/{2k}} ) ( I_n(L)-I_n( L/2 + a/n^{1/{2k}})  )}{f^{n-1}(L/2 + a/n^{1/{2k}} )} da 
\end{align*}
 we get for all $A$ big enough, and for $B_n$ defined in \eqref{u_ge_2}
 \begin{align}
 B_n &\sim \frac{f^{n-1}(L/2)}{n^{3/2k}I_n(L)}\int_{- A}^{0 } h_k(a)h_k(-a) e^{\frac{a^{2k}}{C(f,{2k})^{2k}}}  da \notag \\
 &\sim  \frac{1}{n^{1/k} \frac{1}{k}\Gamma\lt(\frac{1}{{2k}}\rt)C(f,{2k})}\int_{- A}^{0 } h_k(a)h_k(-a)e^{\frac{a^{2k}}{C(f,{2k})^{2k}}}  da 
 \end{align}
 
 Also $$\int_0^{\infty}  h_k(a)h_k(-a)e^{\frac{a^{2k}}{C(f,{2k})^{2k}}}  da = C(f,{2k})^{3} C({2k}) $$
 where $C({2k}) = \int_0^{\infty}  h_{1,k}(a)h_{1,k}(-a)e^{a^{2k}} da  $
 and $ h_{1,k}(x) = \int_{-\infty}^{x} e^{-a^{2k}} da  $ (note that $ C(2k)$ is finite since $ k \ge 2$).
% {\color{red} Je crois que $C(2k)$ se calcule explicitement avec  $\Gamma\lt(\frac{1}{2k}\rt)$, mais je n'y arrive pas, juste une heuristique, ... }
We have for all  $A$ large enough 
\begin{align*}
\int_{- A}^{0 } h(a)h(-a) & e^{\frac{a^{2k}}{C(f,{2k})^{2k}}}  da  +   \frac{c}{A^{2k-2}} (1 - \frac{1}{A^{2k-1}}) (1 - e^{\frac{-A^{2k}}{C(f,{2k})^{2k}}} )  \\
&\le  \liminf_{n\ri\iy} \frac{\EE[\tau_n]}{ \frac{2k}{n^{1/k}\Gamma\lt(\frac{1}{{2k}}\rt)C(f,{2k})}  }  \le  \limsup_{n\ri\iy}\frac{\EE[\tau_n]}{ \frac{2k}{n^{1/k}\Gamma\lt(\frac{1}{{2k}}\rt)C(f,{2k}) }} \\
&\le \frac{c}{A^{2k-2}}  + \int_{- A}^{0 } h(a)h(-a)e^{\frac{a^{2k}}{C(f,{2k})^{2k}}}  da , 
\end{align*} 
and so letting $A$ go to infinity, we get 
\begin{align*}
\EE[ \tau_n] \sim \frac{2kC({2k}) C^2(f,{2k})}{n^{1/k}\Gamma\lt(\frac{1}{{2k}}\rt)} =  \frac{2kC({2k}) }{n^{1/k}\Gamma\lt(\frac{1}{{2k}}\rt)} \lt( \frac{ {(2k}!f(L/2)}{ \vert f^{({2k})}(L/2)\vert} \rt)^{1/k} .
\end{align*}
\end{itemize}
\wwtbp
\begin{rem}
Note that the dominant term of $\EE[ \tau_n]$ comes from $\int_{0}^{ L/2 -A/\sqrt{n} } J_n(s)\, ds $ when $k=1$, and comes from $\int_{L/2 - A/n^{1/{2k}}}^{L/2 }  J_n(s) ds    $ when $k\ge 2 ,$ this essentially leads to two different proofs, despite the apparent similarity. We will find this feature again in the sequel.
\end{rem}
\par
\begin{pro} \label{calvar}
Let $f $ be a $C^{2k+1}$ function on $[0, L]$ satisfying Assumption \eqref{H_f} and \eqref{H_f2}. Assume that for some $  k \ge 2$, $$f(L/2 + h ) = f(L/2) -  \frac{ \vert f^{(2k)}(L/2)\vert}{ (2k)!}  h^{2k} + o(h^{2k})$$ where $f^{(2k)}(L/2) < 0 $ then

 $$\lim_{n\ri\iy}\frac{\EE[\tau_n^2 ] }{\EE[\tau_n ]^2 } = 1 +  \frac{2}{c_k} \int_{- \infty}^{\infty}  h^2_k(-a) e^{\frac{a^{2k}}{C(f,2k)^{2k}}} \int_{- \infty}^{  a} h^2_k(\tilde{a}) e^{\frac{\tilde{a}^{2k}}{C(f,2k)^{2k}}} d\tilde{a}  da,$$
 where $c_k=C(2k)^2 C(f,2k)^{3/k}$, and $C(2k)$ and $ C(f,2k)$ are defined in \eqref{C1} and \eqref{C2}.
 In particular 
 $$ \Var(\tau_n) = O(\EE[\tau_n]^2), \quad \Var(\tau_n) \neq o(\EE[\tau_n]^2).$$
\end{pro}

\begin{proof}
Using Proposition \ref{mean-var} and integration by parts we get
\bq \frac{1}{2} \EE[\tau_n^2 ] &=& \int_0^L  \frac{f^{n-1}(t)}{I_n^2(t)}  \lt(\int_0^t  \frac{I_n^2(s)}{f^{n-1}(s)} u_{n,1}(s) ds \rt) dt. \\
&=& \int_0^L \frac{I_n(s)}{f^{n-1}(s)} \lt( \frac{I_n(L)-I_n(s)}{I_n(L)}\rt) u_{n,1} (s) ds .\\
\eq
Also by Proposition \ref{Poisson} and integration by parts we have
\bq
u_{n,1}(r) &=& \int_r^L  \frac{f^{n-1}(t)}{I_n^2(t)}  \lt(\int_0^t  \frac{I_n^2(s)}{f^{n-1}(s)}  ds \rt) dt \\
&=& \frac{1}{I_n(r)}  \int_0^r  \frac{I_n^2(s)}{f^{n-1}(s)}  ds -  \frac{1}{I_n(L)}  \int_0^L  \frac{I_n^2(s)}{f^{n-1}(s)}  ds + \int_r^L  \frac{I_n(s)}{f^{n-1}(s)}  ds \\
&=& \lt( \frac{I_n(L) - I_n(r)}{I_n(L)I_n(r)} \rt) \int_0^r  \frac{I_n^2(s)}{f^{n-1}(s)}  ds + \int_r^L  \frac{I_n(s)}{f^{n-1}(s)} \lt( \frac{I_n(L) - I_n(s)}{I_n(L)} \rt)  ds .\\ 
\eq
It follows that
\bq \frac{1}{2} \EE[\tau_n^2 ] &=& \int_0^L \frac{1}{f^{n-1}(s)} \lt( \frac{I_n(L)-I_n(s)}{I_n(L)}\rt)^2 
 \int_0^s \frac{I_n^2(t)}{f^{n-1}(t)}  dt  ds \\
&+&\int_0^L \frac{I_n(s)}{f^{n-1}(s)} \lt( \frac{I_n(L)-I_n(s)}{I_n(L)}\rt)  \int_s^L  \frac{I_n(t)}{f^{n-1}(t)} \lt( \frac{I_n(L) - I_n(t)}{I_n(L)} \rt)  dt ds .\\
&=& \int_0^L \frac{1}{f^{n-1}(s)} \lt( \frac{I_n(L)-I_n(s)}{I_n(L)}\rt)^2  \int_0^s \frac{I_n^2(t)}{f^{n-1}(t)}  dt  ds  \\
&+& \frac{1}{2} \lt( \int_0^L  \frac{I_n(t)}{f^{n-1}(t)} \lt( \frac{I_n(L) - I_n(t)}{I_n(L)} \rt)  dt \rt)^2\\
&=& \int_0^L \frac{1}{f^{n-1}(s)} \lt( \frac{I_n(L)-I_n(s)}{I_n(L)}\rt)^2  \int_0^s \frac{I_n^2(t)}{f^{n-1}(t)}  dt  ds  + \frac{1}{2}  \EE[\tau_n]^2, \\
\eq
and so
\bq
\frac{\EE[\tau_n^2 ] }{\EE[\tau_n ]^2 } &=& 1 + \frac{2 \int_0^L \frac{1}{f^{n-1}(s)} \lt( \frac{I_n(L)-I_n(s)}{I_n(L)}\rt)^2  \int_0^s \frac{I_n^2(t)}{f^{n-1}(t)}  dt  ds }{\EE[\tau_n ]^2}\\
&=& 1 + \frac{2 \int_0^L \frac{1}{f^{n-1}(s)} \lt( I_n(L)-I_n(s)\rt)^2  \int_0^s \frac{I_n^2(t)}{f^{n-1}(t)}  dt  ds }{\EE[\tau_n ]^2{I_n(L)}^2}\\
\eq
Note that the numerator in the right hand side of the above first equation is the variance of $ \tau_n$.
Let us write for all $A$ large enough and for all $ n $ large enough:
 
\begin{align}
& &\int_0^L \frac{\lt( I_n(L)-I_n(s)\rt)^2 }{f^{n-1}(s)}  \int_0^s \frac{I_n^2(t)}{f^{n-1}(t)}  dt  ds = \underbrace{\int_0^{L/2- A/n^{1/2k}} \frac{\lt( I_n(L-s)\rt)^2 }{f^{n-1}(s)}  \int_0^s \frac{I_n^2(t)}{f^{n-1}(t)}  dt  ds }_{A_n} \notag \\
 &+& \underbrace{\int_{L/2- A/n^{1/2k}}^{L/2+ A/n^{1/2k}} \frac{\lt( I_n(L-s)\rt)^2 }{f^{n-1}(s)}  \int_0^s \frac{I_n^2(t)}{f^{n-1}(t)}  dt  ds }_{B_n} 
 +  \underbrace{\int_{L/2+ A/n^{1/2k}}^{L} \frac{\lt( I_n(L-s)\rt)^2 }{f^{n-1}(s)}  \int_0^s \frac{I_n^2(t)}{f^{n-1}(t)}  dt  ds }_{C_n} \notag \\ \label{no-cut}
\end{align}

\begin{itemize}
\item Let us start by the term $ A_n$ in \eqref{no-cut}, using \eqref{I_unif_gen}, for $ A$ big enough and for all $n$ sufficiently large we have, for all $s \in [0, L/2 -A/n^{1/2k}], $ 
\begin{equation}\label{I_carre} 
 \frac{f^{n+1}(s)}{n^2 (f'(s))^2}  (1 - \frac{1}{A^{2k-1}})^2   \le \frac{I^2_n(s)}{f^{n-1}(s)}\le \frac{f^{n+1}(s)}{n^2 (f'(s))^2}.
\end{equation}
Let $ W_{n+1}(t)  = \int_0^t  \frac{f^{n+1}(s)}{ (f'(s))^2}ds,$
after integration by parts, we have:
\begin{align*}
W_{n+1}(t) &=  \int_0^t  \frac{f^{n+1}(s)f'(s)}{ (f'(s))^3}ds\\
&= \frac{f^{n+2}(t)}{ (n+2) (f'(t))^3} + 3 \int_0^t\frac{f^{n+2}(s)f''(s)}{ (n+2) (f'(s))^4} ds\\
\end{align*}

Since $f''$ is negative, 
 \begin{equation*}
      W_{n+1}(t) \le \frac{f^{n+2}(t)}{ (n+2) (f'(t))^3}.
 \end{equation*}
It follows that:
\bqn{int_1}
\int_0^s \frac{I^2_n(t)}{f^{n-1}(t)} dt \le \frac{f^{n+2}(s)}{ n^3 (f'(s))^3}.
\eqn
Hence since $ \frac{1}{f'(s)} \sim_{s \ri L/2-} \frac{(2k-1)!}{\vert f^{(2k)}(L/2)\vert (L/2-s)^{ 2k-1}} , $ we have 
\bq
 A_n &\le & \frac{I_n(L)^2}{n^3}\int_0^{L/2- A/n^{1/2k}} \frac{f^{3}(s)}{(f'(s))^3}   ds \\
 &\le & \frac{I_n(L)^2f^{3}(L/2) }{n^3}\int_0^{L/2- A/n^{1/2k}} \frac{1}{(f'(s))^3}   ds \\
 &\sim &  c \frac{I_n(L)^2 }{n^{2/k}A^{6k-4}}. 
\eq
Since by Proposition \ref{prop-mean}, $\EE[ \tau_n] \sim  \frac{c }{n^{1/k}} $, it follows that for all $A$ big enough,
\bqn{C4}
\limsup_{n\ri\iy}  \frac{A_n}{\EE[\tau_n ]^2{I_n(L)}^2}  &\le &\frac{c}{A^{6k-4} } \label{A_n}
\eqn

\item for the term $ B_n$  in \eqref{no-cut}: for $A$ big enough, for all $n$ sufficiently large we write 
\bq
B_n &=& \int_{L/2- A/n^{1/2k}}^{L/2+ A/n^{1/2k}} \frac{\lt( I_n(L-s)\rt)^2 }{f^{n-1}(s)}  \int_0^s \frac{I_n^2(t)}{f^{n-1}(t)}  dt  ds \\
&=& \underbrace{\int_{L/2- A/n^{1/2k}}^{L/2+ A/n^{1/2k}} \frac{\lt( I_n(L-s)\rt)^2 }{f^{n-1}(s)}  \int_0^{L/2- A/n^{1/2k}} \frac{I_n^2(t)}{f^{n-1}(t)}  dt  ds}_{B_n^{(1)}}\\
&+& \underbrace{\int_{L/2- A/n^{1/2k}}^{L/2+ A/n^{1/2k}} \frac{\lt( I_n(L-s)\rt)^2 }{f^{n-1}(s)}  \int_{L/2- A/n^{1/2k}}^s \frac{I_n^2(t)}{f^{n-1}(t)}  dt  ds}_{B_n^{(2)}}.\\
\eq
 \begin{itemize}
 \item Let us estimate the term $B_n^{(1)} $. 
 Using \eqref{C3} and \eqref{int_1} we have,
 \begin{align*}
 \int_0^{L/2- A/n^{1/2k}} \frac{I^2_n(t)}{f^{n-1}(t)} dt &\le \frac{f^{n+2}({L/2- A/n^{1/2k}})}{ n^3 (f'({L/2- A/n^{1/2k}}))^3}
 % \notag 
  \\
 &\sim  c \frac{ f^{n+2}(L/2) e^{\frac{-A^{2k}}{C(f,2k)^{2k}}}}{ n^3(\frac{A^{2k-1}}{ n^{(2k-1)/2k}})^3  }
 % \notag 
 \\
&\sim  c \frac{ f^{n+2}(L/2) e^{\frac{-A^{2k}}{C(f,2k)^{2k}}}}{ n^{3/2k}A^{6k-3}  }. 
% \label{int} 
 \end{align*}
 So using \eqref{concentration}, we have uniformly in  $ a \in [-A,A] $
   \begin{equation} I_n(L/2 + a/n^{1/2k}) \sim  \frac{ f^{n-1}(L/2)}{n^{1/{2k}}} h_k(a). \label{In}
   \end{equation}
   Since  uniformly in  $ a \in [-A,A] $ we have, recall \eqref{C3},   
   \begin{equation} f^n(L/2 - a/n^{1/2k} )\sim f^n(L/2) e^{\frac{-a^{2k}}{C(f,2k)^{2k}}}, \label{fn}
    \end{equation}
   we deduce that
 \bq
 B_n^{(1)} &\le & c \frac{1}{n^{1/2k}} \int_{-A}^{A} \frac{I^2_n(L/2- a/n^{1/2k})}{f^{n-1}( L/2 + a/n^{1/2k})} \frac{ f^{n+2}(L/2) e^{\frac{-A^{2k}}{C(f,2k)^{2k}}}}{ n^{3/2k}A^{6k-3}  } da \\
 & \sim & c  \frac{   f^{n+2}(L/2) e^{\frac{-A^{2k}}{C(f,2k)^{2k}}}}{n^{2/k} A^{6k-3} }\int_{-A}^{A} \frac{I^2_n(L/2- a/n^{1/2k})}{f^{n-1}( L/2 + a/n^{1/2k})} da \\
 & \sim &  c  \frac{   f^{n+2}(L/2) e^{\frac{-A^{2k}}{C(f,2k)^{2k}}}}{n^{2/k} A^{6k-3} }\int_{-A}^{A} \lt( \frac{f^{n-1}(L/2) h_k(-a)}{n^{1/2k}} \rt)^2 \frac{1}{f^{n-1}(L/2) e^{\frac{-a^{2k}}{C(f,2k)^{2k}}}} da \\
 &\sim & c  \frac{   f^{2n+1}(L/2) e^{\frac{-A^{2k}}{C(f,2k)^{2k}}}}{n^{3/k} A^{6k-3} }\int_{-A}^{A} \lt(  h_k(-a)\rt)^2  e^{\frac{a^{2k}}{C(f,2k)^{2k}}} da \\
 & \le & c  \frac{   f^{2n+1}(L/2) }{n^{3/k} A^{6k-3} }\int_{-A}^{A} \lt(  h_k(-a)\rt)^2   da \\
 & \le & 2 c \parallel h_k \parallel_{\infty}^2 \frac{   f^{2n+1}(L/2) }{n^{3/k} A^{6k-4} }. \\
 \eq
Also by \eqref{I_L_ge2} and Proposition \ref{prop-mean} we have $$ \EE[\tau_n ]^2{I_n(L)}^2 \sim \frac{c_kf^{2n-2}(L/2)}{n^{3/k}} .$$
Note that the constant $ c_k = C(2k)^2 C(f,2k)^{3/k}$, where $C(2k)$ and $ C(f,2k)$ are defined in \eqref{C1} and \eqref{C2}.
It  follows that for all $A$ big enough,
\bqn{C5}
\limsup_{n\ri\iy} \frac{B_n^{(1)}}{\EE[\tau_n ]^2{I_n(L)}^2}& \le& \frac{c}{A^{6k-4} } \label{tildB_n}
\eqn

\item Let us estimate the term $B_n^{(2)} $. Using \eqref{fn} and \eqref{In} we have
\bq
B_n^{(2)} & =&  \int_{L/2- A/n^{1/2k}}^{L/2+ A/n^{1/2k}} \frac{\lt( I_n(L-s)\rt)^2 }{f^{n-1}(s)}  \int_{L/2- A/n^{1/2k}}^s \frac{I_n^2(t)}{f^{n-1}(t)}  dt  ds \\
&=& \frac{1}{n^{1/2k}} \int_{- A}^{A} \frac{\lt( I_n( L/2- a/n^{1/2k})\rt)^2 }{f^{n-1}(  L/2+ a/n^{1/2k})}  \int_{L/2- A/n^{1/2k}}^{ L/2 + a/n^{1/2k}} \frac{I_n^2(t)}{f^{n-1}(t)}  dt  da \\
&=& \frac{1}{n^{1/k}} \int_{- A}^{A} \frac{\lt( I_n( L/2- a/n^{1/2k})\rt)^2 }{f^{n-1}(  L/2+ a/n^{1/2k})}  \int_{- A}^{  a} \frac{I_n^2( L/2 + \tilde{a}/n^{1/2k})}{f^{n-1}(L/2 + \tilde{a}/n^{1/2k})}  d\tilde{a}  da \\
& \sim & \frac{1}{n^{1/k}} \int_{- A}^{A} \frac{\lt( \frac{f^{n-1}(L/2)h_k(-a)}{n^{1/2k}}  \rt)^2 }{f^{n-1}(  L/2) e^{\frac{-a^{2k}}{C(f,2k)^{2k}}} }  \int_{- A}^{  a} \frac{ \lt( \frac{f^{n-1}(L/2)h_k(\tilde{a})}{n^{1/2k}}  \rt)^2}{f^{n-1}(  L/2) e^{\frac{-\tilde{a}^{2k}}{C(f,2k)^{2k}}} }  d\tilde{a}  da \\
&\sim & \frac{f^{2n-2}(  L/2)}{n^{3/k}} \int_{- A}^{A}  h^2_k(-a) e^{\frac{a^{2k}}{C(f,2k)^{2k}}} \int_{- A}^{  a} h^2_k(\tilde{a}) e^{\frac{\tilde{a}^{2k}}{C(f,2k)^{2k}}} d\tilde{a}  da
\eq
It follows that for all $A$  big enough
\bqn{C6}
\lim_{n\ri\iy} \frac{B_n^{(2)}}{\EE[\tau_n ]^2{I_n(L)}^2} &= &\frac{ C(A) }{c_k}, \label{barB_n}
\eqn
where $ C(A) = \int_{- A}^{A}  h^2_k(-a) e^{\frac{a^{2k}}{C(f,2k)^{2k}}} \int_{- A}^{  a} h^2_k(\tilde{a}) e^{\frac{\tilde{a}^{2k}}{C(f,2k)^{2k}}} d\tilde{a}  da$
 \end{itemize}
 
 \item for the term $ C_n$  in \eqref{no-cut}: for $A$ big enough, and for all $n$ sufficiently large we write
  
 \bq
 C_n & =& \int_{L/2+ A/n^{1/2k}}^{L} \frac{\lt( I_n(L-s)\rt)^2 }{f^{n-1}(s)}  \int_0^s \frac{I_n^2(t)}{f^{n-1}(t)}  dt  ds  \\
 &=& \underbrace{\int_{L/2+ A/n^{1/2k}}^{L} \frac{\lt( I_n(L-s)\rt)^2 }{f^{n-1}(s)}  \int_0^ {L/2- A/n^{1/2k}}\frac{I_n^2(t)}{f^{n-1}(t)}  dt  ds }_{C_n^{(1)}}\\
 &+&  \underbrace{\int_{L/2+ A/n^{1/2k}}^{L} \frac{\lt( I_n(L-s)\rt)^2 }{f^{n-1}(s)}  \int_{L/2- A/n^{1/2k}}^{L/2+ A/n^{1/2k}} \frac{I_n^2(t)}{f^{n-1}(t)}  dt  ds }_{C_n^{(2)}}\\
 &+& \underbrace{\int_{L/2+ A/n^{1/2k}}^{L} \frac{\lt( I_n(L-s)\rt)^2 }{f^{n-1}(s)}  \int_{L/2+ A/n^{1/2k}}^s \frac{I_n^2(t)}{f^{n-1}(t)}  dt  ds }_{C_n^{(3)}}.
 \eq
 \begin{itemize}
 \item For $C_n^{(1)}$, using  \eqref{int_1}, we have 
 \bq
 C_n^{(1)} =  \lt(\int_0^ {L/2- A/n^{1/2k}}\frac{I_n^2(t)}{f^{n-1}(t)}  dt \rt)^2 \le 
      \lt(  c \frac{ f^{n+2}(L/2) e^{\frac{-A^{2k}}{C(f,2k)^{2k}}}}{ n^{3/2k}A^{6k-3}  }  \rt)^2,
   \eq  and so
    \bqn{C7}
\limsup_{n\ri\iy} \frac{C_n^{(1)}}{\EE[\tau_n ]^2{I_n(L)}^2} & \le &\frac{c}{A^{12k-6} } \label{tildC_n}
\eqn
      
 \item For $C_n^{(2)} $, note that after a change of variable it is equal to $ B_n^{(1)}$
hence
\bqn{C8}
\limsup_{n\ri\iy} \frac{C_n^{(2)}}{\EE[\tau_n ]^2{I_n(L)}^2} &\le& \frac{c}{A^{6k-4} } \eqn

 \item For $C_n^{(3)} $, since $f''\le 0$ and by \eqref{I_carre} we have
 \bq
 C_n^{(3)}  &=& \int_{L/2+ A/n^{1/2k}}^{L} \frac{\lt( I_n(L-s)\rt)^2 }{f^{n-1}(s)}  \int_{L/2+ A/n^{1/2k}}^s \frac{I_n^2(t)}{f^{n-1}(t)}  dt  ds\\
 & = & \int_{0}^{L/2- A/n^{1/2k}} \frac{\lt( I_n(s)\rt)^2 }{f^{n-1}(s)}  \int_{L/2+ A/n^{1/2k}}^{L-s} \frac{I_n^2(t)}{f^{n-1}(t)}  dt  ds\\
 &= & \int_{0}^{L/2- A/n^{1/2k}} \frac{\lt( I_n(s)\rt)^2 }{f^{n-1}(s)}  \int_s^{L/2- A/n^{1/2k}} \frac{I_n^2(L-t)}{f^{n-1}(t)}  dt  ds\\
 &\le & I_n^2(L)\int_{0}^{L/2- A/n^{1/2k}} \frac{\lt( I_n(s)\rt)^2 }{f^{n-1}(s)}  \int_s^{L/2- A/n^{1/2k}} \frac{f'(t)}{f^{n-1}(t)f'(t)}  dt  ds\\
 &\le & \frac{I_n^2(L)}{f'(L/2- A/n^{1/2k})} \int_{0}^{L/2- A/n^{1/2k}} \frac{\lt( I_n(s)\rt)^2 }{f^{n-1}(s)}  \int_s^{L/2- A/n^{1/2k}} \frac{f'(t)}{f^{n-1}(t)}  dt  ds\\
  &\le & \frac{I_n^2(L)}{(n-2)f'(L/2- A/n^{1/2k})} \int_{0}^{L/2- A/n^{1/2k}} \frac{\lt( I_n(s)\rt)^2 }{f^{n-1}(s)}  \frac{1}{f^{n-2}(s)}    ds\\
   &\le & \frac{I_n^2(L)}{(n-2)f'(L/2- A/n^{1/2k})} \int_{0}^{L/2- A/n^{1/2k}} \lt(\frac{ I_n(s) }{f^{n-1}(s)} \rt)^2 f(s)   ds\\
    &\le & \frac{I_n^2(L)}{(n-2)f'(L/2- A/n^{1/2k})} \int_{0}^{L/2- A/n^{1/2k}} \frac{ f^3(s)   }{n^2(f'(s))^2}  ds\\
   &\le & \frac{I_n^2(L)f^3(L/2)}{n^2(n-2)f'(L/2- A/n^{1/2k})} \int_{0}^{L/2- A/n^{1/2k}} \frac{  1  }{(f'(s))^2}  ds\\ 
   &\sim & c \frac{I_n^2(L)f^3(L/2)}{n^3( A/n^{1/2k}))^{2k-1}} \frac{n^{2- 3/2k}}{A^{4k-3}}\\
   &\sim &  c \frac{I_n^2(L)f^3(L/2)}{n^{2/k}A^{6k-4}}.\\
 \eq 
 
 It follows that 
 \bqn{C9}
\limsup_{n\ri\iy} \frac{C_n^{(3)}}{\EE[\tau_n ]^2{I_n(L)}^2} & \le &\frac{c}{A^{6k-4} } 
\eqn
 \end{itemize}
 
 Using equations \eqref{C4},\eqref{C5},\eqref{C6}, \eqref{C7},\eqref{C8} ,\eqref{C9} we get, for all $A$ big enough
 
 \bq
 \frac{ 2 C(A) }{c_k}= \liminf_{n\ri\iy} & & 2 \frac{B_n^{(1)}}{\EE[\tau_n ]^2{I_n(L)}^2} \\ &\le & 
 \liminf_{n\ri\iy}  \frac{2 \int_0^L \frac{1}{f^{n-1}(s)} \lt( I_n(L)-I_n(s)\rt)^2  \int_0^s \frac{I_n^2(t)}{f^{n-1}(t)}  dt  ds }{\EE[\tau_n ]^2{I_n(L)}^2}\\
 &\le& \limsup_{n\ri\iy}\frac{2 \int_0^L \frac{1}{f^{n-1}(s)} \lt( I_n(L)-I_n(s)\rt)^2  \int_0^s \frac{I_n^2(t)}{f^{n-1}(t)}  dt  ds }{\EE[\tau_n ]^2{I_n(L)}^2}\\
 & \le &  2 \limsup_{n\ri\iy}\frac{A_n + B_n^{(2)} + B_n^{(1)} +C_n^{(1)} + C_n^{(2)} + C_n^{(3)}}{\EE[\tau_n ]^2{I_n(L)}^2} \\
&\le & \frac{c}{A^{6k-4} } +   \frac{ 2 C(A) }{c_k}\\
 \eq  
 \end{itemize}
 Passing to the limit when $A$ goes to infinity, and using dominated convergence theorem , we get
 
 \begin{align*}
 & \lim_{n\ri\iy}2  \frac{ \int_0^L \frac{1}{f^{n-1}(s)}   \lt( I_n(L)-I_n(s)\rt)^2  \int_0^s \frac{I_n^2(t)}{f^{n-1}(t)}  dt  ds  }{\EE[\tau_n ]^2{I_n(L)}^2} \\
 & \quad \quad = \frac{2}{c_k} \int_{- \infty}^{\infty}  h^2_k(-a) e^{\frac{a^{2k}}{C(f,2k)^{2k}}} \int_{- \infty}^{  a} h^2_k(\tilde{a}) e^{\frac{\tilde{a}^{2k}}{C(f,2k)^{2k}}} d\tilde{a}  da,
 \end{align*}
 and 
 
 $$\lim_{n\ri\iy}\frac{\EE[\tau_n^2 ] }{\EE[\tau_n ]^2 } = 1 +  \frac{2}{c_k} \int_{- \infty}^{\infty}  h^2_k(-a) e^{\frac{a^{2k}}{C(f,2k)^{2k}}} \int_{- \infty}^{  a} h^2_k(\tilde{a}) e^{\frac{\tilde{a}^{2k}}{C(f,2k)^{2k}}} d\tilde{a}  da,$$ 

It follows that  
\bq\Var(\tau_n)& \sim&  \EE[\tau_n]^2 \Big(\frac{2}{c_k} \int_{- \infty}^{\infty}  h^2_k(-a) e^{\frac{a^{2k}}{C(f,2k)^{2k}}} \int_{- \infty}^{  a} h^2_k(\tilde{a}) e^{\frac{\tilde{a}^{2k}}{C(f,2k)^{2k}}} \Big)\eq
and $ \Var(\tau_n) \neq o(\EE[\tau_n]^2).$\par
\wwtbp
\end{proof}

\begin{theo}\label{nocut}
Let $f $ be a $C^{2k+1}$ function on $[0, L]$  satisfying Assumptions \eqref{H_f}  and \eqref{H_f2}. Assume that for some $  k \ge 2$, \bq
f(L/2 + h ) &=& f(L/2) +  \frac{  f^{(2k)}(L/2)}{ (2k)!}  h^{2k} + o(h^{2k})\eq where $f^{(2k)}(L/2) < 0 $.
For any $n\in\NN\setminus\{1\}$, let  $X_n\df(X_n(t))_{t\geq 0}$ be the Brownian motion described in Definition \ref{BM}. Then the family of diffusion processes $(X_n)_{n\in\NN\setminus\{1\}}$ has no cut-off in separation.
\end{theo}
\begin{proof}
The proof is by contradiction. Suppose that the family of diffusion processes $(X_n)_{n\in\NN\setminus\{1\}}$ has a cut-off in separation with mixing times $(a_n)_{n\in\NN\setminus\{1\}}$. Since for any $n\in\NN\setminus\{1\}$, $\tau_n$ is a sharp strong stationary time for $X_n$, the following convergence  in probability holds for large $n$,
 \bq \frac{\tau_n}{a_n} &\overset{\PP}{\to}& 1 . \eq
 Note that by Propositions \ref{Poisson} and \ref{mean-var}, we have
 $$ \parallel G_n(\un) \parallel_{\infty} = u_{n,1}(0)=  \EE[ \tau_n ], $$
 $$  \parallel G_n \parallel_{\mathbb{L}^{\infty}, \mathbb{L}^{\infty}} = \EE[ \tau_n ],$$
 and
 $$\EE[ \tau_n^k ] = k! G_n^{\circ k}[\un] \le k! \EE[ \tau_n ]^k. $$
 Hence $\frac{\tau_n}{\EE[ \tau_n ]}$ is uniformly integrable.
 Up to  extracting a subsequence, we can suppose that $\frac{\tau_n}{\EE[ \tau_n ]}$ converges in law, says toward a random variable $Y$.
 
 Using Slutsky Theorem we have that $(\frac{\tau_n}{\EE[ \tau_n ]} , \frac{\tau_n}{a_n}) $ converge in law to $(Y,1)$, and so  $ \frac{a_n}{\EE[ \tau_n ]}$ converge in law to $Y$, 
 and there exist $ \lambda \in \mathbb{R}^+$ such that
 $$ \frac{a_n}{\EE[ \tau_n ]} \to \lambda = Y.$$
 Hence $\frac{\tau_n}{\EE[ \tau_n ]} \overset{\PP}{\to}  \lambda ,$ and since $\frac{\tau_n}{\EE[ \tau_n ]}$ is uniformly integrable, the convergence takes place in $ \mathbb{L}^1$, and $ \lambda =1$.
 Also $\frac{\tau_n^2}{\EE[ \tau_n ]^2} \overset{\PP}{\to} 1 ,$
and $\frac{\tau_n^2}{\EE[ \tau_n ]^2}$ is uniformly integrable, so the convergence takes place in $ \mathbb{L}^1$, and 
$$\EE\lt[ \frac{\tau_n^2}{\EE[ \tau_n ]^2} \rt] \to 1,$$
and we get a contradiction with Proposition \ref{calvar}.
\end{proof}

%\wwtbp

%\begin{rem}
 %We expect completely different behaviors when the shape of $f$ looks like a M, in particular we believe that 
 %$\lim_{n\ri\iy}\EE[\tau_n] =+ \infty$.
%\end{rem}

In the next propositions, we show that the cut-off phenomenon in separation occurs  when   $f(x)$ looks like $f(L/2) -C \vert L/2 -x \vert ^{1+\alpha} $ for $x$ near $ L/2$ , with $ \alpha \in (0,1) $, but with a different speed in comparison  with the case $\alpha = 1$ in Theorem  \ref{cut-off-ge}. 

\begin{pro}\label{mean-ge-alpha}
Let $f $ be a $C^2$ function on $[0, L] \setminus \{L/2\}$ and $ C^1$ on $[0, L]$  satisfying Assumptions \eqref{H_f} and \eqref{H_f2}. Assume that locally around $L/2$ we have for some $ \alpha \in (0,1)$ and $C >0$,
 \bq
 f''(L/2- h)&=& -C \vert h \vert^{\alpha -1} + o(\vert h \vert^{\alpha -1})  \eq
  then
$$\EE[ \tau_n] \sim  \frac{2}{n} \int_0^{L/2} \frac{f(s)}{f'(s)} ds$$
\end{pro}
\proof
We follow the same computations  as in the proof of Proposition \ref{mean-ge}, and we adopt the same notations.  For $ A$ big enough and $n$ sufficiently large, let
\begin{align}\label{u_ge_alpha}
\EE[\tau_n] &=2\lt( \underbrace{\int_0^{L/2 - A/n^{1/(1 + \alpha)}}  \frac{I_n(s)}{f^{n-1}(s)} \lt( \frac{I_n(L)-I_n(s)}{I_n(L)}\rt)  ds}_{A_n} \rt. \\
& \quad \quad \quad + \lt. \underbrace{\int_{L/2 - A/n^{1/(1 + \alpha)}}^{L/2 }   \frac{I_n(s)}{f^{n-1}(s)} \lt( \frac{I_n(L)-I_n(s)}{I_n(L)}\rt)  ds}_{B_n} \rt)\label{C10}
\end{align}
Since $f'(L/2)= 0$ we have,
 $$ f'(L/2 - h) = \frac{C}{\alpha} \sign(h) \vert h \vert^{\alpha} + o (\vert h \vert^{\alpha}), $$
 $$ f(L/2 - h) = f(L/2) - \frac{C}{\alpha(\alpha +1)}  \vert h \vert^{ 1 + \alpha} + o (\vert h \vert^{1+\alpha}). $$
 Let $ C_{\alpha} =  \frac{C}{\alpha(\alpha +1)}$,  since $ \ln (f(L/2 - h) ) = \ln (f(L/2)) - \frac{C_{\alpha}}{f(L/2)} \vert h \vert^{ 1 + \alpha} + o (\vert h \vert^{1+\alpha})  ,$
 using Laplace's method we get:
\bqn{C11} I_n(L/2) &= &\int_0^{\frac{L}{2}} f^{n-1}(t)dt  \sim_{n \ri \infty}   \frac{f^{n-1}(L/2)}{n^{1/(1 + \alpha)}} C(\alpha, f), \eqn
 where $C(\alpha, f) = (\frac{f(L/2)}{C_{\alpha}})^{1/(1 + \alpha)} \frac{1}{1 + \alpha}\Gamma( \frac{1}{1+\alpha}) .$
\begin{itemize}
\item For $ A$ big enough and for all  $n$ sufficiently large, and for  $ s \in [0,L/2 - A/n^{1/(1 + \alpha)}]$, let us compute an equivalent of $A_n$ in \eqref{u_ge_alpha}.
 
 Let $ m_n = \inf_{[0,L/2 - A/n^{1/(1 + \alpha)}]} f''$
 by hypothesis on $f$, we have for large $n$,
 $$ m_n \sim -C \frac{A^{\alpha - 1}}{n^{(\alpha - 1)/(1 + \alpha)}}.$$
  Since 
 $$\frac{m_n f(s)}{(n+1)(f'(s))^2} \ge \frac{m_n f(L/2)}{(n+1)(f'(L/2 - A/n^{1/(1 + \alpha)}))^2} \sim_{n\ri \iy} - \frac{c_1}{A^{1+\alpha}} , $$
 we get that for $A$ big enough and $n$ sufficiently large
 $$\frac{m_n f(s)}{(n+1)(f'(s))^2} \ge -\frac{1}{A^{(1+\alpha) /2}}. $$
 Also for large $n$,$$\frac{f^n(s)}{nf'(s)I_n(L)} \le \frac{f^n( L/2 - A/n^{1/(1 + \alpha)})}{n f'(L/2 - A/n^{1/(1 + \alpha)}) I_n(L)}  \sim c_1 \frac{e^{-\frac{C_{\alpha}}{f(L/2)}A^{1+\alpha}} }{A^{\alpha}},$$
 so  for $A$ big enough and $n$ sufficiently large
 $$\frac{f^n(s)}{nf'(s)I_n(L)} \le  \frac{1}
 % {e^{-\frac{C_{\alpha}}{f(L/2)}A^{1+\alpha}} }
 {A^{\alpha/2}}.$$
 Using \eqref{I_born} and \eqref{vol_born}, it follows that for $A$ big enough and $n$ sufficiently large,  uniformly in   $ s \in [0,L/2 - A/n^{1/(1 + \alpha)}]$,
 \begin{equation}\label{I_born_alpha}
 \frac{f^n(s)}{nf'(s)} \lt( 1  - \frac{1}{A^{(1+\alpha) /2}} \rt)   \le I_n(s)\le \frac{f^n(s)}{nf'(s)},
\end{equation}
 and so 
 \begin{equation}\label{J_unif_alpha}
 \frac{f(s)}{nf'(s)}  \lt(1 -\frac{1 }{A^{\alpha/2}} \rt)  \lt( 1  - \frac{1}{A^{(1+\alpha) /2}}  \rt)  \le J_n(s)\le \frac{f(s)}{nf'(s)}.
\end{equation}
where $J_n$ is defined in \eqref{J_def_ge}.\par
Since $\frac{f(s)}{f'(s)} \sim_{s \ri (L/2)_{-}} \frac{\alpha}{C} \frac{f(L/2)}{\vert L/2-s \vert^\alpha} $, is integrable at $ L/2$ we get that for all $A$ big enough:
   \begin{equation} \label{sup_inf_alpha}
 \lt(1 -\frac{1 }{A^{\alpha/2}} \rt)  \lt( 1  - \frac{1}{A^{(1+\alpha) /2}}  \rt) \le \liminf_{n\ri\iy}  \frac{A_n}{\frac{1}{n} \int_0^{L/2} \frac{f(s)}{f'(s)} ds } \le \limsup_{n\ri\iy}  \frac{A_n}{\frac{1}{n} \int_0^{L/2} \frac{f(s)}{f'(s)} ds }\le 1  
   \end{equation} 
   
  \item  For $ A$ big enough and for all  $n$ sufficiently large, and for  $ s \in [L/2 - A/n^{1/(1 + \alpha)}, L/2]$, let us compute the equivalent of $B_n$ in \eqref{C10}.
  
 More generally when $s \in [\frac{L}{2}-\frac{A}{n^{1/(1 + \alpha)}}, \frac{L}{2}+\frac{A}{n^{1/(1 + \alpha)}}]$,  write $s = L/2 + a/n^{1/(1 + \alpha)} $, with $a \in [-A,A] $. We have uniformly in $a \in [-A,A] $:
\begin{align}
I_n(L/2 & + a/n^{1/(1 + \alpha)} ) = I_n(L/2) + \int_{L/2}^{L/2 + a/n^{1/(1 + \alpha)}} f^{n-1}(x)dx \notag \\
&= I_n(L/2) + \frac{1}{n^{1/(1 + \alpha)}} \int_{0}^{a} f^{n-1}(L/2 + \frac{h}{n^{1/(1 + \alpha)}})dh \notag \\
&= I_n(L/2) + \frac{1}{n^{1/(1 + \alpha)}}\int_{0}^{a} \lt(f\lt(\f{L}2\rt) - \frac{C_{\alpha}\vert h \vert ^{1+\alpha}}{n} +o\lt(\f{\vert h \vert ^{1+\alpha}}{n}\rt) \rt)^{n-1}dh \notag \\
 &= I_n(L/2) + \frac{f(L/2)^{n-1}}{n^{1/(1 + \alpha)}}
 \int_{0}^{a} e^{(n-1) \ln \lt(1 - \frac{C_{\alpha}\vert h \vert ^{1+\alpha}}{f(L/2)n} 
+o\lt(\f{\vert h \vert ^{1+\alpha}}{n}
\rt)\rt)}dh 
\notag \\
 &\sim \frac{ f^{n-1}(L/2)}{n^{1/(1 + \alpha)}} \lt( C(\alpha, f)  + \int_0^a e^{- \frac{C_{\alpha}\vert h \vert^{1+\alpha}}{f(L/2)}  } dh   \rt) \notag \\
 &= \frac{ f^{n-1}(L/2)}{n^{1/(1 + \alpha)}}   \int_{-\infty }^a e^{- \frac{C_{\alpha}\vert h \vert^{1+\alpha}}{f(L/2)} } dh  .\label{lemme_alpha}
\end{align} 
Concerning the justification of the equivalent in the above computation, 
%first remark that the dominant term $\frac{ f^{n-1}(L/2)}{n^{1/(1 + \alpha)}}  $ is the same in the sum in the left hand side of  the fourth equality, also 
note the integral term  
$ \int_{0}^{a} e^{(n-1) \ln (1 - \frac{C_{\alpha}\vert h \vert ^{1+\alpha}}{f(L/2)n} +o(\vert h \vert ^{1+\alpha}/n) ) }dh $ converges for fixed $a$ to $\int_0^a e^{- \frac{C_{\alpha}h^{1+\alpha}}{f(L/2)}  } dh $, by the dominated convergence theorem (since the integrand is bounded by $1$), finally by Dini's theorem this convergence is  uniform in $ a \in [-A,A]$. The last equality follows by a change of variable formula that shows that  $C(\alpha, f)  = \int_{-\infty}^0 e^{- \frac{C_{\alpha}\vert h \vert^{1+\alpha}}{f(L/2)}  } dh $.
%{\color{red} si vous �tes ok avec la preuve ci-dessus on peut par le m�me argument all�ger la r�gularit� de $f$ (de $1$ ) dans tous ce qui pr�c�de.}  
%\LM{Il faudrait effectivement relaxer toutes les hypoth�ses de r�gularit� des premiers r�sultats, mais on n'a pas envie de le faire !} {\color{red}  peut �tre en remarque ?}%%%

Define $ h_{\alpha}(a) =\int_{-\infty }^a e^{- \frac{C_{\alpha}\vert h \vert^{1+\alpha}}{f(L/2)} } dh  $,  we get that for  $A$ big enough, and for $B_n$ defined in \eqref{C10}
 \begin{align}
 B_n &\df \int_{L/2 - A/n^{1/(1 + \alpha)}}^{L/2 }   \frac{I_n(s)}{f^{n-1}(s)} \lt( \frac{I_n(L)-I_n(s)}{I_n(L)}\rt)  ds \notag \\
 &\sim \frac{1}{n^{3/(1 + \alpha)}I_n(L)} \int_{-A}^{0} \frac{h_{\alpha}(a)h_{\alpha}(-a) f^{2n-2}(L/2)}  {f^{n-1}(L/2)} e^{\frac{C_{\alpha}\vert a \vert^{1+\alpha}}{f(L/2)} }   da \notag \\
 &\sim \frac{c}{n^{2/(1 + \alpha)}} = o\lt(\frac{1}{n}\rt) ,
 \end{align}
 where we took \eqref{C11} into account.
Hence using \eqref{sup_inf_alpha} we have for all  $A$ large enough 
$$   \lt(1 -\frac{1 }{A^{\alpha/2}} \rt)  \lt( 1  - \frac{1}{A^{(1+\alpha) /2}}  \rt)\le  \liminf_{n\ri\iy} \frac{\EE[\tau_n]}{ \frac{2}{n} \int_0^{L/2} \frac{f(s)}{f'(s)} ds }   \le  \limsup_{n\ri\iy}\frac{\EE[\tau_n]}{ \frac{2}{n} \int_0^{L/2} \frac{f(s)}{f'(s)} ds } \le 1  , $$ 
and so leting $ A$ tends to infinity we get
$$\EE[ \tau_n] \sim  \frac{2}{n} \int_0^{L/2} \frac{f(s)}{f'(s)} ds.$$ 
   
\end{itemize} 
\wwtbp
\begin{pro}\label{var-ge-alpha}
Let $f $ be a $C^2$ function on $[0, L] \setminus \{L/2\}$ and $ C^1$ on $[0, L]$  satisfying Assumptions \eqref{H_f} and \eqref{H_f2}. Assume that  for some $ \alpha \in (0,1)$ and $C >0$, we have for all $\vert h\vert >0$ small enough,
 $$f''(L/2- h)= -C \vert h \vert^{\alpha -1} + o(\vert h \vert^{\alpha -1})  $$ then
\bq \Var(\tau_n) & = & o\lt(\frac{1}{n^2}\rt). \eq
\end{pro}
\proof
From Proposition \ref{var} and after integration by parts, we have that  for all $A$ large enough and for all $ n $ large enough:
\begin{align}
 \frac{\Var(\tau_n)}{2}  &= \underbrace{\int_0^{L/2- A/n^{1/(1 + \alpha)}} J_n(s) (u_{n,1}'(s))^2 ds}_{A_n}
 + \underbrace{\int_{L/2- A/n^{1/(1 + \alpha)}}^{L/2+ A/n^{1/(1 + \alpha)}}J_n(s) (u_{n,1}'(s))^2 ds}_{B_n} \label{**_ge_k1_alpha0} \\
 & + \underbrace{\int_{L/2+ A/n^{1/(1 + \alpha)}}^{L} J_n(s) (u_{n,1}'(s))^2 ds}_{C_n}, \label{**_ge_k1_alpha} \\
 \notag
\end{align}
where $J_n$ is defined in \eqref{J_def_ge} and $(u_{n,1}')^2$ in \eqref{dv}.
% \bqn{dv_alpha}(u_{n,1}'(t))^2 &= &\lt( \frac{f^{n-1}(t)}{I^2_n(t)} \int_0^t \frac{I^2_n(s)}{f^{n-1}(s)} ds \rt)^2\eqn
\begin{itemize}
\item  Let us start by estimating the term $ A_n$, using \eqref{I_born_alpha}, it follows that for  $ A$ big enough and for all $n$ sufficiently large, and for all $s \in [0, L/2 -A/n^{1/(1 + \alpha)} ], $ 
\begin{equation}\label{21_ge_alpha}
\frac{f^{n+1}(s)}{n^2 (f'(s))^2} \lt(1 -\frac{1 }{A^{(1+\alpha)/2}} \rt) ^2 \le \frac{I^2_n(s)}{f^{n-1}(s)} \le \frac{f^{n+1}(s)}{n^2 (f'(s))^2}.
 \end{equation}
Since $f'' \le 0$ in $[0, L/2[$,  \eqref{Int_I/f^} holds for all $ 0 \le t < L/2$,
% \begin{equation}\label{Int_I/f^_alpha}
% \int_0^t \frac{I^2_n(s)}{f^{n-1}(s)} ds   \le  \frac{f^{n+2}(t)}{n^3 (f'(t))^3},
% \end{equation}
hence  for  $ A$ big enough and for all $n$ sufficiently large, and for all $ t\in[0, L/2 -A/n^{1/(1 + \alpha)} ]$:
$$(u_{n,1}'(t))^2 \le \lt(\frac{f(t)}{nf'(t) (1-\frac{1}{A^{(1+\alpha)/2}})^2 }  \rt)^2 .$$
Also by \eqref{21_ge},  
$$ J_n(s) \df \frac{I_n(s)}{f^{n-1}(s)} \frac{I_n(L-s)}{I_n(L)} \le \frac{f(s)}{n f'(s)}. $$
 Hence for $A$ big enough, for all $n$ sufficiently large, and for $A_n$ defined in \eqref{**_ge_k1_alpha0}, we have
\begin{align*}
A_n &\le \frac{1}{n^3 (1 - \frac{1}{A^{(1+\alpha)/2}})^4} \int_0^{L/2- A/n^{1/(1 + \alpha)}}\frac{f^{3}(s)}{(f'(s))^3} ds\\
& \le \frac{f(L/2)^3}{n^3 (1 - \frac{1}{A^{(1+\alpha)/2}})^4} \int_0^{L/2- A/n^{1/(1 + \alpha)}}\frac{1}{(f'(s))^3} ds
% \\
% & = O(\frac{1}{ n^{4/(1+\alpha)}}),
\end{align*}
Taking into account that $ \frac{1}{f'(s)} \sim_{s \ri L/2-} \frac{c}{\vert(L/2-s) \vert^{\alpha}} $, we get
\bq
\int_0^{L/2- A/n^{1/(1 + \alpha)}}\frac{1}{(f'(s))^3} ds&\sim&
\lt\{\begin{array}{ll}
c&\hbox{, if $3\alpha<1$}\\
c\ln(n^{1/(1 + \alpha)}/A)&\hbox{, if $3\alpha=1$}\\
c\frac{n^{(3\alpha - 1)/(1+\alpha)}}{A^{(3\alpha -1)}}&\hbox{, if $3\alpha>1$.}
\end{array}\rt.\eq
Hence  for $A $ big enough, we get
\bq
A_n&\sim&
\lt\{\begin{array}{ll}
\di {c}/{n^3}&\hbox{, if $3\alpha<1$}\\[2mm]
\di{c\ln(n)}/{n^3}&\hbox{, if $3\alpha=1$}\\[2mm]
\di{c}/{n^{\f{4}{1+\alpha}}}&\hbox{, if $3\alpha>1$,}
\end{array}\rt.\eq
in particular,
since $ \alpha \in (0,1)$,  
\bqn{compl2_alpha} A_n &=& o\lt(\frac{1}{n^2}\rt).\eqn

\item For the term $B_n $ in \eqref{**_ge_k1_alpha0}: for $A$ big enough, for all $n$  large enough and  for $ a \in [-A,A] $, let $x = L/2 +a/ n^{1/(1 + \alpha)} $ we have
\begin{align*}
 & \lt\vert u_{n,1}'\lt(\frac{L}{2}  + \frac{a}{n^{1/(1 + \alpha)} }\rt)\rt\vert \\
& =  \frac{f^{n-1}\lt(\frac{L}{2} +\frac{a}{n^{1/(1 + \alpha)} }\rt)}{I^2_n\lt(\frac{L}{2} +\frac{a}{n^{1/(1 + \alpha)} }\rt)} \lt(\int_0^{L/2-A/n^{1/(1 + \alpha)}} \frac{I^2_n(s)}{f^{n-1}(s)} ds  +  \int_{L/2-A/n^{1/(1 + \alpha)}}^{L/2+\frac{a}{n^{1/(1 + \alpha)} }} \frac{I^2_n(s)}{f^{n-1}(s)} ds  \rt) .
\end{align*}

By the above computation and \eqref{Int_I/f^}, we have 

 $$ \int_0^{L/2-A/n^{1/(1 + \alpha)}} \frac{I^2_n(s)}{f^{n-1}(s)} ds \le \frac{f^{n+2}(L/2-\frac{A}{n^{1/(1 + \alpha)} })}{n^3 (f'(L/2-\frac{A}{n^{1/(1 + \alpha)} }))^3} \sim c\frac{f^{n+2}(L/2)e^{-\frac{C_{\alpha}}{f(L/2)}A^{(1 + \alpha)}}}{n^{3/(1+ \alpha)}A^{3\alpha}}.$$

Recall that from \eqref{lemme_alpha}, and for $h_{\alpha} (a) = \int_{-\infty }^a e^{- \frac{C_{\alpha}\vert h \vert^{1+\alpha}}{f(L/2)} } dh$, we have uniformly over  $ a \in [-A,A] $
   \begin{equation} I_n(L/2 + a/n^{1/(1 + \alpha)}) \sim f^{n-1}(L/2)\f{h_{\alpha}(a)}{n^{1/(1 + \alpha)}}, \label{10_gen_k1_alpha} 
   \end{equation}
   hence uniformly in $ a \in [-A,A] $, 
   
\begin{align*}
 \int_{L/2-A/n^{1/(1 + \alpha)}}^{L/2+\frac{a}{n^{1/(1 + \alpha)} }} \frac{I^2_n(s)}{f^{n-1}(s)} ds  &=  \frac{1}{n^{1/(1 + \alpha)}}\int_{-A}^{a} \frac{I^2_n(L/2+\frac{\tilde{a}}{n^{1/(1 + \alpha)} })}{f^{n-1}(L/2+\frac{\tilde{a}}{n^{1/(1 + \alpha)} })} d \tilde{a}\\
 & \sim  \frac{f^{n-1}(L/2)}{n^{3/(1 + \alpha)}} \int_{-A}^{a} h_{\alpha}^2(\tilde{a})e^{\frac{C_{\alpha}}{f(L/2)} \vert \tilde{a}\vert^{(1 + \alpha)}} d \tilde{a} \\
 &=\frac{f^{n-1}(L/2)\theta_{\alpha,A}(a)}{n^{3/(1+ \alpha)}},\\
\end{align*}
where $\theta_{\alpha,A}(a) \df \int_{-A}^{a} h_{\alpha}^2(\tilde{a})e^{\frac{C_{\alpha}}{f(L/2)} \vert \tilde{a} \vert^{(1 + \alpha)}} d \tilde{a}  $.

 Since uniformly in $ a \in [-A,A] $,
 
 $$ f^n(L/2 - a/n^{1/(1 + \alpha)} )\sim f^n(L/2)e^{-\frac{C_{\alpha}}{f(L/2)}\vert a \vert^{(1 + \alpha)}} ,$$

 we have  for $A$ big enough  
 $$\lt\vert u_{n,1}'\lt(\frac{L}{2} +\frac{a}{n^{1/(1 + \alpha)} }\rt) \rt\vert \le c \frac{e^{-\frac{C_{\alpha}}{f(L/2)}\vert a \vert ^{(1 + \alpha)}}} {n^{1/(1 + \alpha)}h_{\alpha}^2(a)} \lt(\theta_{\alpha,A}(a) + \frac{1}{A^{3\alpha}}\rt).$$
Hence 
\begin{align*}
 B_n &\df \int_{L/2- A/n^{1/(1 + \alpha)}}^{L/2+ A/n^{1/(1 + \alpha)}}J_n(s) (u_{n,1}'(s))^2 ds\\
 &= \frac{1}{n^{1/(1 + \alpha)}}\int_{-A}^{A}J_n(L/2+ a/n^{1/(1 + \alpha)}) (u_{n,1}'(L/2+ a/n^{1/(1 + \alpha)}))^2 da\\
 &\le \frac{1}{n^{1/(1 + \alpha)}}\int_{-A}^{A}J_n(L/2+ a/n^{1/(1 + \alpha)}) \lt(c \frac{e^{-\frac{C_{\alpha}}{f(L/2)}\vert a \vert^{(1 + \alpha)}}} {n^{1/(1 + \alpha)}h_\alpha^2(a)} \lt(\theta_{\alpha,A}(a) + \frac{1}{A^{3\alpha}}\rt)\rt)^2 da\\
 &\sim  \frac{c(A)}{n^{4/(1+\alpha)} },\\
\end{align*}
where we use in the third line that
\begin{align*} J_n(L/2 + a/n^{1/(1 + \alpha)}) & =\frac{I_n(L/2 +a/n^{1/(1 + \alpha)})}{f^{n-1}(L/2 +a/n^{1/(1 + \alpha)})}  \frac{I_n(L/2 -a/n^{1/(1 + \alpha)})}{I_n(L)} \\
&\sim c \frac{h_{\alpha}(a)h_{\alpha}(-a)e^{\frac{C_{\alpha}}{f(L/2)}\vert a \vert^{(1 + \alpha)}}}{n^{1/(1 + \alpha)} } ,
\end{align*}
 and $c(A) $ is a constant that depends on $A$.
It follows that for all $A$ large enough,  
\bqn{compl5_alpha} B_n &=&   o\lt(\frac{1}{n^2}\rt).\eqn

\item For the last term $C_n $ in \eqref{**_ge_k1_alpha}, note that $J_n(s)=J_n(L -s)$ so
\begin{align*}
C_n &= \int_{L/2+ A/n^{1/(1 + \alpha)}}^{L} J_n(s) (u_{n,1}'(s))^2 ds\\
&=\int_0^{L/2- A/n^{1/(1 + \alpha)}} J_n(s) (u_{n,1}'(L-s))^2 ds.\\
\end{align*}
Also for $ s \le L/2- A/n^{1/(1 + \alpha)} $
\begin{align*}
\vert u_{n,1}'(L - s) \vert &=  \frac{f^{n-1}(s)}{I^2_n(L-s)} \int_0^{L-s} \frac{I^2_n(t)}{f^{n-1}(t)} dt \\
&= \frac{f^{n-1}(s)}{I^2_n(L-s)} \lt( \int_0^{L/2-A/n^{1/(1 + \alpha)}} \frac{I^2_n(t)}{f^{n-1}(t)} dt  \rt. \\
& \quad \quad + \lt. \int_{L/2-A/n^{1/(1 + \alpha)}}^{L/2+A/n^{1/(1 + \alpha)}} \frac{I^2_n(t)}{f^{n-1}(t)} dt +
\int_{L/2+A/n^{1/(1 + \alpha)}}^{L-s} \frac{I^2_n(t)}{f^{n-1}(t)} dt \rt).
\end{align*}
The first two terms in the above bracket has been computed in the above item,  for the last term since for $s\leq L/2-A/n^{1/(1 + \alpha)}$,

\begin{align*}
\int_{L/2+A/n^{1/(1 + \alpha)}}^{L-s} \frac{I^2_n(t)}{f^{n-1}(t)} dt
&\le I^2_n(L-s)\int_{L/2+A/n^{1/(1 + \alpha)}}^{L-s} \frac{1}{f^{n-1}(t)} dt\\
&=  I^2_n(L-s)\int_s^{L/2-A/n^{1/(1 + \alpha)}}\frac{1}{f^{n-1}(t)} dt\\
&=  I^2_n(L-s)\int_s^{L/2-A/n^{1/(1 + \alpha)}}\frac{f'(t)}{f^{n-1}(t)f'(t)} dt\\
&\le \frac{I^2_n(L-s)}{f'(L/2-A/n^{1/(1 + \alpha)})} \lt(\frac{f^{-n +2}(s)}{n-2} \rt) \\
&\sim c\frac{f^{-n+2}(s)I^2_n(L-s)}{n^{1/(1+\alpha)}A^{\alpha}}\\
\end{align*}
Since $L -s \ge L/2 $, we have for $A$ big enough and for all $n$ sufficiently large,
\begin{align*}
\vert u_{n,1}'(L - s)\vert &\le  \frac{f^{2n-2}(L/2)}{n^{3/(1+\alpha)}I_n^2(L/2)}\lt(  \frac{1}{A^{3\alpha}} + \theta_{A}(A)   \rt) +  c \frac{f(s)}{n^{1/(1+\alpha)}A^{\alpha}}\\
 &\le   \frac{c}{n^{1/(1+\alpha)}} \lt(  \frac{1}{A^{3\alpha}} + \theta_{A}(A)   \rt)
 +c \frac{f(s)}{n^{1/(1+\alpha)}A^{\alpha}}
 \le   \frac{c(A)}{n^{1/(1+\alpha)}} 
\end{align*}
where in the second inequality, we used \eqref{C11}. Since, for $ s \le L/2 - A/n^{1/(1 + \alpha)} $, \eqref{I_born_alpha} yield  $ J_n(s) \le \frac{f(s)}{n f'(s)} $, so for $A$ big enough and for all $n$ sufficiently large,
\begin{align*}
C_n 
&=\int_0^{L/2- A/n^{1/(1 + \alpha)}} J_n(s) (u_{n,1}'(L-s))^2 ds.\\
&\le \frac{c^2(A)}{n^{2/(1 + \alpha)}} \int_0^{{L/2- A/n^{1/(1 + \alpha)}} } J_n(s)ds\\
&\le  \frac{c^2(A) f(L/2)}{n^{1+2/(1 + \alpha)}} \int_0^{{L/2- A/n^{1/(1 + \alpha)}} } \frac{1}{f'(s)} ds\\
&\sim  c \frac{c^2(A) f(L/2)}{n^{1+2/(1 + \alpha)}}  = o( \frac{1}{n^2}),
\end{align*}
where for the equivalent we use that $f'(s) \sim_{(L/2)^{-}} c \vert L/2-s\vert^{\alpha}  $ and the integral is convergent. The last equality follows since $ \alpha \in (0,1)$.
Hence \bqn{compl6_alpha}C _n &= &  o\lt(\frac{1}{n^2}\rt).\eqn
\end{itemize}
Putting \eqref{compl2_alpha}, \eqref{compl5_alpha} and \eqref{compl6_alpha} together, we deduce that $ \Var(\tau_n)= o\lt(\frac{1}{n^2}\rt) .$
\wwtbp

\begin{theo}\label{cut-off-ge_alpha}
Let $f $ be a $C^2$ function on $[0, L] \setminus \{L/2\}$ and $ C^1$ on $[0, L]$  satisfying Assumptions \eqref{H_f} and \eqref{H_f2}.
Assume that  for some $ \alpha \in (0,1)$ and $C >0$, we have for all $\vert h\vert >0$ small enough,
\bq
f''(L/2- h)&=& -C \vert h \vert^{\alpha -1} + o(\vert h \vert^{\alpha -1}).  \eq
Let  $X_n\df(X_n(t))_{t\geq 0}$ be the Brownian motion described in Definition \ref{BM}.
Then the family of diffusion processes $(X_n)_{n\in\NN\setminus\{1\}}$ has a cut-off in separation with mixing times $(a_n)_{n\in\NN\setminus\{1\}} =  \lt(\frac{2}{n} \int_0^{L/2} \frac{f(s)}{f'(s)} ds\rt)_{n\in\NN\setminus\{1\}}$, in the sense of Section \ref{cutoff}. 
\end{theo}

\proof
Use Theorem \ref{th1},  Proposition \ref{mean-ge-alpha} and Proposition \ref{var-ge-alpha}.
\wwtbp

%\begin{cor}\label{cor_ker}
%Let $f $ be a $C^2$ function on $[0, L] \setminus \{L/2\}$ and $ C^1$ on $[0, L]$  satisfying Assumptions \eqref{H_f} and \eqref{H_f2}.
%Assume that  for some $ \alpha \in (0,1)$ and $C >0$, we have for all $\vert h\vert >0$ small enough,
%\bq
%f''(L/2- h)&=& -C \vert h \vert^{\alpha -1} + o(\vert h \vert^{\alpha -1}).  \eq
%
%For $n\in\NN\setminus\{1\}$,  consider  the Brownian motion $X_n\df(X_n(t))_{t\geq 0}$ in $M_f^n$ described in Definition \ref{BM}.
% There exist $C >0$ and $n_0 \in \mathbb{N} $ such that for all  $ r > 0$ and for all $ n \ge n_0$,
\begin{cor}
With same hypothesis as in Theorem \ref{cut-off-ge_alpha}, there exist $\tilde{C} >0$ and $n_0 \in \mathbb{N} $ such that for all  $ r > 0$, $0< r' < 1 $ and for all $ n \ge n_0$,
 \bq
   \lVe \mathcal{L}\lt(X_n({(1+r){\frac{2}{n} \int_0^{L/2} \frac{f(s)}{f'(s)} ds }}) \rt) - \mathcal{U}_n \rVe_{\mathrm{tv}} &\le &\frac{\tilde{C}}{r^2 n^{\frac{1-\alpha}{1+\alpha}}}\\
\fo y \in M_f^n,\qquad    P^{(n)}_{(1+r)\frac{2}{n} \int_0^{L/2} \frac{f(s)}{f'(s)} ds} ( \tilde{0},y)&\ge &\lt(1- \frac{\tilde{C}}{r^2 n^{\frac{1-\alpha}{1+\alpha}}} \rt) \frac{1}{\Vol(M_f^n )}\\
\inf_{ y \in M_f^n}   P^{(n)}_{(1-r')\frac{2}{n} \int_0^{L/2} \frac{f(s)}{f'(s)} ds} ( \tilde{0},y)&\le &\lt( \frac{\tilde{C}}{r'^2 n^{\frac{1-\alpha}{1+\alpha}}} \rt) \frac{1}{\Vol(M_f^n )}
 \eq
 \end{cor}
% the same conclusion as in Corollary \ref{cor_ker} could be give,
\begin{proof}
  In the proof of Proposition \ref{var-ge-alpha}  we have in fact (since the dominant term is $C_n$) $$\Var(\tau_n) =  O\lt(\frac{1}{n^{1+\frac{2}{1+\alpha}}}\rt). $$
The result follows with the same proof as the proof of Corollary  \ref{cor_ker}.
% we get that
%there exist $C >0$ and $n_0 \in \mathbb{N} $ such that for all  $ r > 0$ and for all $ n \ge n_0$,
% \bq
%   \lVe \mathcal{L}\lt(X_n({(1+r){\frac{2}{n} \int_0^{L/2} \frac{f(s)}{f'(s)} ds }}) \rt) - \mathcal{U}_n \rVe_{\mathrm{tv}} &\le &\frac{C}{r^2 n^{\frac{1-\alpha}{1+\alpha}}}\\
%\fo y \in M_f^n,\qquad    P^{(n)}_{(1+r)\frac{2}{n} \int_0^{L/2} \frac{f(s)}{f'(s)} ds} ( \tilde{0},y)&\ge &\lt(1- \frac{C}{r^2 n^{\frac{1-\alpha}{1+\alpha}}} \rt) \frac{1}{vol(M_f^n )}
% \eq
\end{proof}
\begin{theo}\label{cut-off-ge_alpha2}
Let $f $ be a $C^2$ function on $[0, L] \setminus \{L/2\}$ and $ C^1$ on $[0, L]$  satisfying Assumptions \eqref{H_f} and \eqref{H_f2}. Assume that  for some $ \alpha >1$ and $C >0$, we have for all $\vert h\vert >0$ small enough,
\bq
f''(L/2- h)&=& -C \vert h \vert^{\alpha -1} + o(\vert h \vert^{\alpha -1}).  \eq
Let  $X_n\df(X_n(t))_{t\geq 0}$ be the Brownian motion described in Definition \ref{BM}.
Then the family of diffusion processes $(X_n)_{n\in\NN\setminus\{1\}}$ has  no cut-off in separation. 
\end{theo}

\proof
Following the proof of Proposition \ref{mean-ge-alpha}, 
% take $ m_n = \inf_{[0 ,L]} f''$ by changing $k$ by $1+\frac{\alpha}{2}$, since $ \frac{1}{s^{1+\alpha}}$ is not integrable at $0$ using \eqref{J_unif_alpha}, 
we  show that
\bq\EE[\tau_n] &\sim &\frac{C_\alpha(f)}{n^{2/(2+\alpha)}} .\eq
Following the proof of Proposition \ref{calvar}, we  show that $\Var(\tau_n)/2$ is
 equivalent for $n$ large to $\frac{B_n^{(2)}}{I_n(L)^2} $ with the same decompositions as introduced there. It follows that
 $\Var(\tau_n)/\EE[\tau_n]^2$ converges toward a positive constant and we conclude as in 
  Theorem \ref{nocut}. 
\wwtbp

To end the paper, let us give the
\prooff{Proof of Theorem \ref{theo1}}
The items of Theorem \ref{theo1} correspond respectively to 
 Theorem \ref{cut-off-ge_alpha}, 
 Theorem \ref{cut-off-ge} and 
Theorem \ref{cut-off-ge_alpha2}.
\wwtbp

\vskip2cm
\hskip70mm
\vbox{
\copy4
 \vskip5mm
 \copy5
  \vskip5mm
 \copy6
}

\end{document}